\documentclass[11pt,notitlepage,twoside,a4paper]{amsart}
 \usepackage{amsfonts}

\usepackage{amsmath,amssymb,enumerate}

\usepackage{epsfig,fancyhdr,color}

\usepackage{amssymb}
\usepackage{amsmath,amsthm}
\usepackage{latexsym}
\usepackage{amscd}
\usepackage{psfrag}
\usepackage{graphicx}
\usepackage[latin1]{inputenc}
\usepackage[all]{xy}
\usepackage[mathcal]{eucal}

\definecolor{NoteColor}{rgb}{1,0,0}

\renewcommand{\textsc}{\textcolor{red}}

\newtheorem{theorem}{\rm\bf Theorem}[section]
\newtheorem{proposition}[theorem]{\rm\bf Proposition}
\newtheorem{lemma}[theorem]{\rm\bf Lemma}
\newtheorem{corollary}[theorem]{\rm\bf Corollary}
\newtheorem*{theorem 1}{\rm\bf Proposition 1}
\newtheorem*{theorem 2}{\rm\bf Proposition 2}

\theoremstyle{definition}
\newtheorem{definition}[theorem]{\rm\bf Definition}

\theoremstyle{remark}

\def\interieur#1{\mathord{\mathop{\kern 0pt #1}\limits^\circ}}

\title[Ideal triangles, hyperbolic surfaces and the Thurston metric]{Ideal triangles, hyperbolic surfaces and the Thurston metric on Teichm\"uller space}

\author[Athanase Papadopoulos]{Athanase Papadopoulos}

\thanks{}

 \address{Institut de Recherche Math\'ematique Avanc\'ee\\ CNRS et Universit\'e de Strasbourg\\\small 7 rue Ren\'e
  Descartes - 67084 Strasbourg Cedex, France}

\date{\today}

\makeindex
\begin{document}
\maketitle
\begin{abstract}
These are notes on the hyperbolic geometry of surfaces, Teichmüller spaces and Thurston's metric on these spaces.  They are associated with lectures I gave at the Morningside Center of Mathematics of the Chinese Academy of Sciences in March 2019 and at the Chebyshev Laboratory of the Saint Petersburg State University in May 2019.  In particular, I survey several results on the behavior of stretch lines, a distinguished class of geodesics for Thurston's metric and I point out several analogies between this metric and Teichm\"uller's metric. Several open questions are addressed. The final version of these notes will appear in the book \emph{Moduli Spaces and Locally Symmetric Spaces} edited by L. Ji and S.-T. Yau.

\medskip


\medskip

\end{abstract}
 \section{Introduction}

These notes are associated with lectures I gave at the Morningside Center of Mathematics of the Chinese Academy of Sciences in March 2019 and at the Chebyshev Laboratory of the Saint Petersburg State University in May 2019. The aim of these lectures was to provide the students with an introduction to Thurston's metric on Teichm\"uller space, preceded by some necessary background on Thurston's theory on surfaces.  

Thurston's metric on Teichm\"uller space provides a point of view for the study of deformations of Riemann surfaces in which hyperbolic geometry plays a central role, as opposed to the study of Teichm\"uller's metric  which involves in an essential way complex analysis, in particular the theory of quasiconformal mappings. It is however useful to recall right at the beginning that despite the fact that the two points of view are different, there are several formal analogies between them, and before starting a formal survey of the theory, we invoke some of them:

(1) There is a striking correspondence between, on the one hand, two different definitions of the Thurston metric (one definition based on the smallest Lipschitz constant of homeomorphisms between two marked hyperbolic surfaces, and another one that uses the comparison between the length spectra of simple closed geodesics), and on the other hand, two definitions of the Teichm\"uller metric (one definition that uses the smallest  distorsion of quasiconformal homeomorphisms between two conformal structures,  and another one based on the comparison of extremal lengths of simple closed curves). We shall review all these definitions in \S \ref{s:metric} below. 

 (2)  Thurston's metric involves the use of pairs consisting of a maximal geodesic lamination and a transverse measured foliation that define geodesics for this metric, and the Teichm\"uller metric involves pairs of transverse measured foliations that define geodesics for it. This is reviewed  in \S \ref{s:metric} below.
 
 (3)  The definition of the infinitesimal Finsler structures for these metrics is based on the logarithmic derivative of the hyperbolic  length function in the case of the Thurston metric, and on the logarithmic derivative of the extremal length function in the case of the Teichm\"uller metric. This is also reviewed  in \S \ref{s:metric} below.
 
  (4) There are results on the asymptotic behavior of stretch lines that are comparable to those on the asymptotic behavior of Teichm\"uller geodesics. This is reviewed in \S\S \ref{s:stretch},   \ref{s:pairs} and \ref{s:anti} below.
  
  There are many other analogies between the two metrics, some of them still awaiting for an explanation and leading to open problems.

In these notes, the main results I survey concerning Thurston's metric, after recalling the basic material, concern the positive and the negative limiting behavior of stretch lines. These lines are equipped with a natural parametrization for which they are geodesics for Thurston's metric. Their definition is based on the construction of Lipschitz maps between ideal hyperbolic triangles. The reason for which a substantial part of this survey is dedicated to these lines is that their study and that of their limiting behavior use several fundamental aspects of Thurston's theory on surfaces, including geodesic laminations, measured foliations, the boundary structure of Teichm\"uller space and others. The investigation of  the asymptotic behavior of stretch lines involves a series of estimates on the length functions of simple closed geodesics and of geodesic laminations associated with hyperbolic structures tending to infinity in Teichm\"uller space, and a comparison between these length functions and intersection functions associated with measured foliations, called horocyclic foliations We review all these topics before surveying the main results.  
   
The plan of the next sections is the following: \S \ref{s:notation} contains notation and background material on the hyperbolic geometry of surfaces, including a study of triangles, ideal triangles,  area, horocycles, Teichm\"uller spaces, geodesic laminations, measured foliations and Thurston's compactification of Teichm\"uller space. In \S \ref{s:ideal}, we study hyperbolic surfaces obtained by gluing ideal triangles. 
   \S \ref{s:metric} contains an introduction to Thurston's metric on Teichm\"uller space and an exposition of its basic geometric features, in particular its completeness and its Finsler structure. In \S \ref{s:stretch}, we survey the properties of stretch lines and their limiting behavior. In \S \ref{s:pairs}, we study the relative asymptotic behavior of pairs of stretch lines, namely, questions of parallelism and divergence between such  pairs.
  \S \ref{s:anti} is concerned with the asymptotic behavior of anti-stretch lines (that is, stretch lines traversed in the reverse direction). In general, these lines are not geodesics for Thurston's metric; they are geodesics for another metric which we call the negative Thurston metric. The results surveyed in \S\S  \ \ref{s:pairs} and \ref{s:anti} are due to Guillaume Th\'eret. 
We conclude, in \S  \ref{s:conclusion}, with some remarks and open questions.
   
   Thurston introduced his metric in 1986 \cite{Thurston1986}, and the theory became gradually an active research field. A first survey, with a comparison of the results on this metric with those of the Teichm\"uller metric, appeared in 2007 \cite{PT}. A set of open problems on Thurston's metric appeared in 2015 \cite{S}, after a conference held on this topic at the American Institute of mathematics in Palo Alto. Recently, several new results and directions related to this metric were obtained, see in particular the papers \cite{AD, DGK, DLRT, HS, HP1, HP2, HS}. This renewal of interest in this topic was the motivation for giving these lectures and writing these notes.

\section{Hyperbolic geometry and Teichm\"uller spaces} \label{s:notation}

This section contains an introduction to some basic notions of hyperbolic geometry and related matters that will be useful in the rest of this survey. We start with some standard terminology and notation.

\subsection{Hyperbolic structures}\label{ss:hyp} We shall denote the hyperbolic plane by $\mathbb{H}^2$.
It is sometimes practical to use the upper half-plane model of this plane. We recall that this is the subset  $\{(x,y)\in \mathbb{R}^2, \ y>0 \}$ of the Euclidean plane $\mathbb{R}^2$ endowed with the $x,y$ coordinates, equipped with an infinitesimal length element at each point given by 
\[ds= \frac{\sqrt{dx^2+dy^2}}{y}.\]
This means that for any parametrized $C^1$ curve $\gamma:[a,b]\to \mathbb{H}^2$, where $[a,b]$ is a compact interval in $\mathbb{R}$ and where in coordinates $\gamma(t)=(x(t),y(t))$, one can compute its length $L(\gamma)$
by the formula
\[L(\gamma)=\int_a^b\frac{\sqrt{dx^2(t)+dy^2(t)}}{y(t)}dt
.\]
A metric is then defined on $\mathbb{H}^2$ by setting the distance between two arbitrary points to be the infimum of the length of all $C^1$ curves joining them. Such a metric, where the distance between two points is equal to the infimum of lengths of paths 
joining them, is called a \emph{length metric}.

We recall that a \emph{geodesic} in a set $X$ equipped with a distance function $d$ is a map $\gamma:I\to X$, where $I$ is an interval of $\mathbb{R}$, satisfying $d(\gamma(t_1),\gamma(t_3))= 
d(\gamma(t_1),\gamma(t_2))+ d(\gamma(t_2),\gamma(t_3))$ for any  $t_1\leq t_2\leq t_3$ in $I$. If the interval $I$ is compact, then the geodesic is called a \emph{geodesic segment}, or, more simply,  a \emph{segment}. If $I=[0,\infty)$, then the geodesic is called a \emph{geodesic ray}, or a \emph{ray}.
We shall often identify a geodesic with its image.

It is well known that a geodesic in the upper half-plane model of the hyperbolic plane is either the intersection of this upper half-plane with a Euclidean circle whose center is on the line $y=0$, or the intersection of the upper half-plane with a vertical Euclidean line. 

In the upper half-plane model $\mathbb{H}^2$, the isometry group of the hyperbolic plane is identified with the projective special linear group $\mathrm{PSL}(2,\mathbb{R})$ acting on $\mathbb{H}^2$ by M\"obius transformations, that is, transformations of the form  
\[z \mapsto \frac{az+b}{cz+d}, \ a,b,c,d\in \mathbb{R}, \ ad-bc=1
.\]
This action induces a transitive action on the unit tangent bundle of $\mathbb{H}^2$, that is, on pairs consisting of a point in $\mathbb{H}^2$ and a unit vector in the tangent space to $\mathbb{H}^2$ at that point.

At one place in these notes, we shall also consider the disc model of the hyperbolic plane  for the purpose of better visualizing symmetry (Figure \ref{fig:horo} below). This is the open unit disc in the complex plane equipped at each point with the infinitesimal length element given, in the$(x,y)$ coordinates, by 
\[ds=2 \frac{\sqrt{dx^2+dy^2}}{1-(x^2+y^2)}.\] 
The distance is defined as above, as the length distance associated with this infinitesimal length element. A geodesic in this model is either the intersection of the unit disc with a Euclidean circle perpendicular to the boundary circle, or a diameter.
In this model, the boundary at infinity is the unit circle.

We mention however that all of hyperbolic geometry can be developed without any model. This is done e.g. in Lobachevsky's \emph{Pangeometry} \cite{Lobachevsky}. For a  recent model-free introduction to hyperbolic geometry, including the derivation of the trigonometric formulae, the reader can refer to the lecture notes \cite{AP}. Working in some models may simplify some calculations, especially because it makes use of the underlying more familiar Euclidean geometry, but we find it aesthetically less appealing.  Lobachevsky made extensive computations (in particular, elaborate computations of area and volume) without using any model.

The upper half-plane and the disc models of the hyperbolic plane are conformal in the sense that in each of these models the notion of angle between any two lines at a point where they intersect coincides with the notion of angle in the underlying Euclidean plane.

 In what follows, $S$ is an oriented surface of finite topological type, that is, $S$ is obtained from a compact oriented surface without boundary by deleting a finite number of points. A \emph{hyperbolic structure} on $S$ is a maximal atlas 
 $\{(U_i,\phi_i)\}_{i\in\mathcal{I}}$ 
where  for each $i\in\mathcal{I}$,
 $U_i $ is an open subset  of $S$   and
$\phi_i$ a homemorphism from $U_i$ onto an open subset of the
hyperbolic plane, satisfying $\displaystyle \bigcup_{i\in\mathcal{I}} U_i=S$ and such that any coordinate change map of the form
$\phi_i\circ\phi_j^{-1}$ is, on each connected component of $\phi_j(U_i\cap U_j)$,
 the restriction of an orientation-preserving isometry of $\mathbb{H}^2$. A pair $(U_i,\phi_i)$ is called a \emph{chart} of the atlas. $U_i$ is a chart domain and $\phi_i$ a chart map.

  We shall assume that $S$ has negative Euler characteristic, so that it admits at least one 
a hyperbolic structure.

A surface equipped with a hyperbolic structure carries a 
  length metric
obtained by taking on each
chart domain $U_i$ the pull-back by the associated chart map  $\phi_i$ of
the metric on the image $\phi_i(U_i)$ induced from its inclusion in $\mathbb{H}^2$.
The metrics obtained on the various open sets $U_i$ of $S$ provide a consistent way of measuring lengths of piecewise $C^1$ paths in $S$, and this naturally leads to a length metric on $S$, called a \emph{hyperbolic metric}. The surface $S$ equipped with this metric is called a \emph{hyperbolic surface}.

All the hyperbolic metrics that we shall consider on $S$ will be complete and of finite area. This means that every puncture of $S$ is a \emph{cusp}, that is,  it has a neighborhood in the surface which, geometrically, is an annulus obtained as the quotient of a region of the upper half-plane model of hyperbolic space of the form $\{y\geq a\}$, where $a$ is some positive constant, by the cyclic group generated by the isometry $z\mapsto z+1$. More intrinsically, such an annulus is isometric to the quotient of a horodisc in the hyperbolic plane by a parabolic transformation preserving this horodisc. (We shall recall the notion of horodisc in \S \ref{ss:horocycles} below.)  

At some places, we shall consider hyperbolic surfaces with boundary, and in this case all the boundary components will be closed geodesics.

In the rest of these notes, we denote by $\mathcal{S}$ the set of  homotopy classes of essential simple closed curves on $S$, that is, simple closed curves that are neither homotopic to a point nor to a puncture. A theorem that originates in the work of Hadamard \cite{Hadamard} says that each element of $S$ is represented by a unique closed geodesic on the surface. This is one of the main building blocks of the theory of hyperbolic geometry of surfaces with non-trivial fundamental group.

\subsection{Teichm\"uller space}\label{ss:T} The Teichm\"uller space  $\mathcal{T}=\mathcal{T}(S)$ of $S$ is the set of isotopy classes of hyperbolic metrics on $S$. It is equipped with a topology which can be defined in several equivalent ways. We now recall one of them.

  We start with the map from $\mathcal{T}=\mathcal{T}(S)$  to the space $\mathbb{R}_+^{\mathcal{S}}$ of positive functions on $\mathcal{S}$ which associates to each element $g$ of  $\mathcal{T}(S)$ the function $\gamma\mapsto l_g(\gamma)$ on $\mathcal{S}$, 
where  $l_g(\gamma)$ denotes the length of the unique geodesic on $S$ representing the homotopy class $\gamma$ with respect to the hyperbolic metric $g$. This map $\mathcal{T}\to \mathbb{R}_+^{\mathcal{S}}$  is injective. This is the famous result stating that a hyperbolic surface is determined by its marked simple length spectrum. Finally, we equip $\mathcal{T}=\mathcal{T}(S)$ with the topology induced on the image of this embedding by the weak topology on the space $\mathbb{R}_+^{\mathcal{S}}$. The details are the subject of  \cite[\S1.4]{FLP}.  

With this topology, the Teichm\"uller space of a surface of genus $g\geq 2$ and with $n\geq 0$ punctures is homeomorphic to $\mathbb{R}^{6g-6+2n}$. 

An explicit way of obtaining a homeomorphism  $\mathcal{T}(S)\to\mathbb{R}^{6g-6+2n}$ is to use the  parametrization of $\mathcal{T}(S)$ by the so-called Fenchel--Nielsen coordinates. This parametrization depends on the choice of a maximal collection of disjoint simple closed geodesics (or, equivalently, a pair of pants decomposition of $S$). To each such geodesic are then associated two parameters: one parameter being its length (an element in $(0,\infty)$), and the other one being a twist parameter (an element in $\mathbb{R}$) which measures how the pairs of pants on the two sides of the geodesic (the two pair of pants may be equal) are glued together in the surface. The twist parameter depends on the choice of an origin for the gluing. A careful description of the Fenchel-Nielsen parameters is contained in Thurston's book \cite[p. 271]{Thurston-Book}. The relation between the two topologies of  $\mathcal{T}(S)$ that we just recalled follows from \cite[Expos\'e 7]{FLP}.

In \S \ref{s:metric}, we shall consider an asymmetric metric on $\mathcal{T}(S)$. The topology of Teichm\"uller space is also the one induced by this metric, with an appropriate definition for the topology induced by an asymmetric metric. We shall discuss this in detail in \S \ref{s:metric}. 
 
\subsection{Boundary}

The hyperbolic plane $\mathbb{H}^2$ has a natural boundary, denoted by $\partial\mathbb{H}^2$. It is obtained by adjoining to $\mathbb{H}^2$  the space of equivalence classes of asymptotic geodesic rays. In this context, two geodesic rays $r_1:[0,\infty)\to \mathbb{H}^2$ and $r_2:[0,\infty)\to \mathbb{H}^2$ are said to be \emph{asymptotic} if their images are at a bounded distance from each other, that is, if there exists a constant $C$ satisfying $d(r_1(t),r_2(t))<C$ for any $t$ in $[0,\infty)$ where $d$ is the hyperbolic distance. Seen from an arbitrary point of $\mathbb{H}^2$, the boundary of this space may be identified with the space of (endpoints of) geodesic rays  starting at that point.  (This follows from the fact that there are no asymptotic rays that start at the same point.) The hyperbolic metric does not extend to the boundary, but the topology does. When the boundary is seen as the set of endpoints of geodesic rays starting at a basepoint, the space union its boundary is equipped with the topology obtained by adding to each geodesic ray its endpoint (or point at infinity). There are several ways of defining formally this topology by providing a sub-basis for the open sets, and to see that it does not depend on the choice of the basepoint. One possible choice of such a sub-basis is the union of the open sets of $\mathbb{H}^2$ together with a new open set for each open half-plane in $\mathbb{H}^2$. These new open sets represent a chart near the boundary: one declares that a point in $\mathbb{H}^2$ belongs to this set if, as a point of the hyperbolic plane, it belongs to the corresponding half-plane, and that a point on the boundary $\partial \mathbb{H}^2$  belongs to this set if and only if it can be represented by a geodesic ray that lies entirely in the given half-plane. With this topology, the union $\overline{\mathbb{H}^2}= \mathbb{H}^2\cup\partial\mathbb{H}^2$ is homeomorphic to a closed disc. In the upper half-plane model, the boundary of the hyperbolic plane is the union of the line $y=0$ with the point at infinity, equipped with a topology that makes this union homeomorphic to a circle.  

For any two distinct points  in $\partial\mathbb{H}^2$, there is a unique geodesic in $\mathbb{H}^2$ having these two points as endpoints. We say that the geodesic \emph{joins} the two points.

The isometry group of the hyperbolic plane acts triply transitively on the boundary of this plane, that is, it acts transitively on ordered triples of distinct points on this boundary (but it is not true that the isometry group acts triply transitively on the boundary when the latter is seen as the set of endpoints of rays starting at a given point!). The triple transitivity of this action can be proved using linear algebra and the identification we recalled in \S \ref{ss:hyp} of the isometry group of the upper half-space model of the hyperbolic plane with the group $\mathrm{PSL}(2,\mathbb{R})$, but it can also be easily deduced from the axioms of hyperbolic geometry.

\subsection{Ideal triangles} 
Given three distinct points on the boundary of the hyperbolic plane,  an \emph{ideal triangle} having these three points as vertices is the closed subset of the hyperbolic plane bounded by the three geodesics that join pairwise these three points at  infinity. Alternatively, one may define the ideal triangle as the convex hull of the three distinct points at infinity.

Any two ideal triangles are isometric. This  can be deduced from the fact that the isometry group of the hyperbolic plane acts triply transitively on the boundary of this space, but it can also be deduced from more basic principles, as we now explain.

 We start with the classical fact that in hyperbolic geometry the isometry type of a triangle is determined by its three angles (see e.g. Lobachevsky's treatise, in which he gives a formula for the side of a hyperbolic triangle in terms of the three angles \cite[p. 29]{Lobachevsky}). We may extend this result (by continuity) to ideal triangles, as follows. The three angles of an ideal triangle are zero. (One may see this in the upper half-plane model of the hyperbolic plane where the geodesics are circles perpendicular to the boundary. When two such geodesics meet at the same point at the boundary they make a zero angle there.) We deduce from these two facts that any two ideal triangles, since they have equal angles, are isometric.

One of the themes with which we shall be acquainted in these notes is that ideal triangles are simpler to deal with than the usual triangles. Ideal triangles are extensively used in the deformation theory of hyperbolic surfaces developed by Thurston in the paper \cite{Thurston1986}. In particular, it is very easy to construct Lipschitz maps between ideal triangles, and this is a basic tool in Thurston's metric theory of Teichm\"uller space.

We mention finally that ideal triangles are objects in hyperbolic geometry that have no analogues in the two other geometries of constant curvature (Euclidean and spherical).

\subsection{Area}

The area of a triangle in the hyperbolic plane is equal to its \emph{angle deficit}, that is, the deficit to $\pi$ of the sum of its three angles. In other words, if $A,B,C$ are the three angles of a hyperbolic triangle, then its area is equal to $\pi-(A+B+C)$.  This can be seen as a consequence of the Gauss--Bonnet theorem. We recall that this theorem applies to any geodesic triangle $\Delta$ on an arbitrary differentiable surface $F$ equipped with Riemannian metric. Here, a geodesic triangle on $F$ is a simply-connected subset of $F$ which is bounded by three geodesic segments. The theorem says that if these three segments make among themselves interior angles $A,B,C$, then we have 
\[\int\int_{\Delta}Kd\sigma = A+B+C-\pi,\]
where $K$ is the Gaussian curvature and $d\sigma$ is the area element on $F$. The hyperbolic plane is a surface equipped with a Riemannian metric of constant Gaussian curvature $=-1$. Thus, in this setting, the above formula holds with $K=-1$ and the integral on the left hand side in the above equation is the area of the triangle, with a  negative sign. This proves the desired area formula.

There is a more intuitive and elementary view on the formula for the area of a triangle as its angle deficit, which we give now.

Let us start by proving the following:
\begin{proposition}
In the hyperbolic plane, the angle deficit of a triangle is always positive.
\end{proposition}
\begin{proof}
 We must prove that the sum of the three angles in any triangle is $<\pi$. The argument is by contradiction. We start with the fact that there exists, in the hyperbolic plane, a triangle whose angle sum is $<\pi$. For this, it suffices to consider an ideal triangle (whose three angles are zero), or, if one prefers non-ideal triangle, we can take a triangle whose vertices are near infinity, close to those of an ideal triangle (then its three angles will be close to zero). Assume now that there exists a triangle in the hyperbolic plane whose angle sum is $\geq \pi$. Then, by continuously moving the three vertices of one of the two triangles that we have towards those of the other, we can find a triangle whose angle sum is equal to $\pi$. By taking a countable number of copies of such a triangle, we make a tiling of the plane which is combinatorially modeled on the tiling of the Euclidean plane by isometric triangles represented in Figure \ref{fig:tiling}.
\begin{figure}[htbp]
\centering
\includegraphics[width=0.5\linewidth]{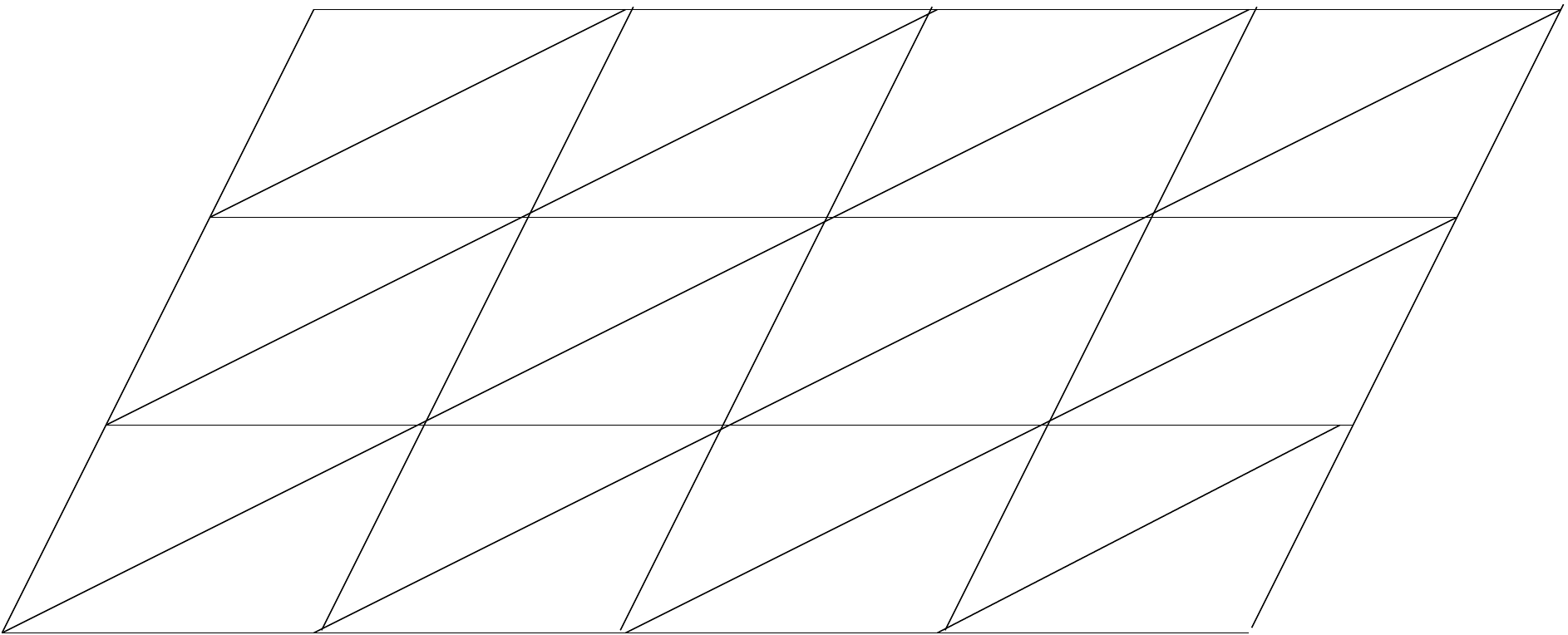}    \caption{\small {A tiling of the Euclidean plane by isometric triangles}}   \label{fig:tiling}  
\end{figure}
 From the assumption on the angle sum, the total sum of the angles around each vertex is equal to two right angles, so this combinatorial tiling gives a genuine tiling of the hyperbolic plane. But the existence of such a tiling of the hyperbolic plane is impossible. The reason is that it would give us two equidistant geodesics in the hyperbolic plane, contradicting one of the fundamental principles of hyperbolic geometry. (It is known that the existence of equidistant geodesics in Euclidean geometry is equivalent to the parallel axiom, but hyperbolic geometry is precisely Euclidean geometry modified so that the parallel axiom is replaced by its negation.)

We conclude that the angle deficit of any triangle in the hyperbolic plane is strictly positive. 
\end{proof}

Now we can use this fact to define a notion of area. The argument that we give is a variation on an argument contained in Lambert's treatise \cite{Lambert-Blanchard}. 

First of all, we have to agree on the definition of an area function: this is a function defined on figures, which is additive under the operation of taking disjoint unions, or, equivalently, under the operation of subdivision. The main examples of figures are polygons, which can be decomposed into triangles. More general figures are obtained as limits of polygons. 

With this in mind, we see that to show that angle deficit is a good notion of  area it is sufficient to show the following
\begin{proposition}
In the hyperbolic plane, angle deficit of triangles is invariant under subdivision of triangles. 
\end{proposition}
\begin{proof}
Let us take an arbitrary triangle  in the hyperbolic plane and let us subdivide it as in Figure \ref{fig:addition} into two smaller triangles with angles $A,E,F$ and $B,C,D$.  The angle deficits of the two smaller triangles are $\pi-(A+E+F)$ and $\pi-(B+C+D)$.  Adding these two quantities and using the fact that $D+E=\pi$, we find $2\pi-(A+E+F+B+C+D)= \pi-(A+B+C+F)$, which is the angle deficit of the large triangle. We can subdivide the triangle in various ways and we get the same result. Thus, angle deficit in hyperbolic geometry is additive on triangles, which is what we wanted to prove.
     \end{proof}
     In fact, with a little bit of care regarding the axioms that an area function should satisfy, one can prove that in hyperbolic geometry the only reasonable area function is, up to a multiple constant, the one that assigns to each triangle its angle deficit.

\begin{figure}[htbp]
\centering
 \psfrag{A}{\small $A$}
  \psfrag{B}{\small $B$}
 \psfrag{C}{\small $C$}
\psfrag{D}{\small $D$}
  \psfrag{E}{\small $E$}
 \psfrag{F}{\small $F$}
\includegraphics[width=0.4\linewidth]{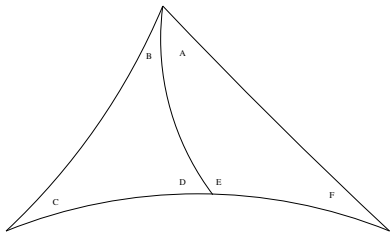}    \caption{\small {A triangle subdivided into two triangles}}   \label{fig:addition}  
\end{figure}

We can apply the preceding result to the area of ideal triangles: since  the three angles of an ideal triangle are all zero, its angle deficit, i.e. its area, is equal to $\pi$.

 \subsection{Horocycles} \label{ss:horocycles} A horocycle in the hyperbolic plane is a limit of circles in the following sense: Consider a geodesic ray $r:[0,\infty)\to \mathbb{H}^2$ starting at a point $P=r(0)$. On this geodesic ray, for every $t>0$, consider the circles $C_t$ centered at $r(t)$ and passing through $P$. As $t\to\infty$, the family of circles $C_t$ converge, in the Hausdorff topology of the space of closed subsets of the compactified hyperbolic plane $\overline{\mathbb{H}^2}$, to a subset $\mathcal{H}$ homeomorphic to a circle, which is called a \emph{horocycle}. 
 
 We note incidentally that horocyles were already singled out by Lobachevsky, see \cite[p. 7]{Lobachevsky}, who called them limit circles and used them extensively in his work.

 The horocycle $\mathcal{H}$ obtained in the above construction intersects the boundary $\partial\mathbb{H}^2$ in a single point called its \emph{center}, which is the limit of the family of centers of the circles $C_t$. It coincides with the limit point $r(\infty)$ of the geodesic ray $r$. By varying the point $P$ on a geodesic $\gamma:(-\infty,\infty)\to\mathbb{H}^2$ and making the above construction with the family of all sub-geodesic rays of $\gamma$ pointing in the same direction, we obtain a foliation of $\mathbb{H}^2$ by geodesics centered at the point $\gamma(\infty)$ (see Figure \ref{fig:horo} below).

A \emph{horodisc} is obtained using an analogous construction, as the limit of discs $D_t$ centered at $r(t)$ and passing through $r(0)$, instead of the limit of circles $C_t$. The discs $D_t$ in the compactified hyperbolic plane are the filled-in circles $C_t$, an the horodiscs are the filled-in horocycles.

The horocycle $\mathcal{H}$ meets perpendicularly the geodesic ray $r(t)$ that was used to define it. 
 In the upper half-plane model of the hyperbolic plane, the horocycles are either the Euclidean circles that are tangent to the real axis $y=0$ (and their center is this point of tangency) or the Euclidean lines of equation $y=C$ with $C>0$, that is, the Euclidean lines that are parallel to the real axis (and their center is the point $\infty$ in this model).  
 
Using the formula that we recalled in \S \ref{ss:hyp} for the length element in the upper half-plane model, it is easy to compute the length of a piece of horocycle contained in a horizontal line at height $y$ in that plane, joining the two points of coordinates $(a,y)$ and $(b,y)$ with $a<b$. This length is:
\[\int_a^b \sqrt{\frac{dx^2+dy^2}{y^2}}=\int_a^b\frac{dx}{y}=\frac{b-a}{y}.\]
This formula is useful in the computation of the dilatation constant of Thurston's stretch maps between ideal triangles and the resulting stretch lines in Teichm\"uller space. We shall consider stretch lines in \S \ref{s:metric} and \S \ref{s:stretch} below.

\subsection{Geodesic laminations}\label{ss:laminations}
A geodesic lamination on $S$ is a closed subset of this surface which is the union of disjoint simple geodesics. These geodesics are called the \emph{leaves} of $\mu$. They can be either closed geodesics or bi-infinite geodesics.

 Thurston showed the importance of geodesic laminations in the deformation theory of hyperbolic surfaces, in Teichm\"uller theory and in the theory of hyperbolic 3-manifolds.  In his Princeton lecture notes, he introduced the set $\mathcal{GL}$ of geodesic laminations on a hyperbolic surface. He considered two topologies on this space, namely, the Hausdorff topology and the geometric topology (See \cite[Chapter 8, \S 8.1]{Thurston-Princeton}).  The latter is also called the Thurston topology. Besides Thurston's original notes, references on geodesic laminations include the books \cite{CB} and \cite{CEG}.

 In what follows, we shall use individual geodesic laminations, and we shall not deal with the space $\mathcal{GL}$ of geodesic laminations, but we shall work with a space of geodesic laminations with transverse measure, namely, the space of measured geodesic laminations. This is the subject of the next subsection.

\subsection{Measured geodesic laminations}\label{ss:laminations}
 
 A \emph{measured geodesic lamination} is a geodesic lamination equipped with a transverse measure, that is, a measure on transversals (arcs that are transverse to the leaves of this lamination). This measure is assumed to be finite on compact transversals and invariant by the local holonomy maps, that is, the isotopies (continuous motions) of the transversals that keep each point on the same leaf. We also ask that the support of the measure on each transversal coincides with the intersection of this arc with the lamination. The underlying geodesic lamination,  without its transverse measure, is said to be the \emph{support} of the measured geodesic lamination.

A measured geodesic lamination on $S$ is said to be \emph{compactly supported} if its support is compact. In particular, such a measured geodesic  lamination cannot contain infinite leaves with an end converging to a cusp.  

We denote by $\mathcal{ML}$ the set of compactly supported measured geodesic laminations on $S$.  This space is equipped with a natural topology obtained by embedding it into the space of nonnegative functions on $\mathcal{S}$. The embedding is obtained by associating to each element $\mu$ of  $\mathcal{ML}$ the intersection function $\gamma\mapsto i(\mu,\gamma)$ 
where  $i(\mu,\gamma)$ denotes the total transverse measure relative to $\mu$  of the unique geodesic representative of $\gamma$ with respect to the underlying hyperbolic structure. The transverse measure is understood to be equal to zero if the geodesic is not transverse to $\mu$ (that is, if it is a leaf of $\mu$). 
This embedding of $\mathcal{ML}$ into the function space $\mathbb{R}_{\geq 0}^{\mathcal{S}}$ equipped with the product topology induces a topology on $\mathcal{ML}$ which we call the
 \emph{measure topology}.

A simple closed geodesic $\gamma$ on a hyperbolic surface is considered to be a measured geodesic lamination equipped with the counting measure, that is,  the measure which assigns to each transverse arc the number of intersection points of that arc with $\gamma$. In this way, we have a canonical injection $\mathcal{S}\subset \mathcal{ML}$. This injection can be extended in an obvious way to an injection of the set of positively weighted simple closed curves into the space of geodesic measured laminations: 
\[\iota:\mathbb{R}_+\times \mathcal{S}\to \mathcal{ML}.\]
We have the following
\begin{proposition}\label{prop:extension}
The image of the map $\iota$ is dense in $\mathcal{ML}$.
\end{proposition}
The proof is given in \cite[Corollary 4.5]{FLP} in the setting of measured foliations instead of measured laminations which amounts to the same; see the exposition of measured foliations in the next section.

We note finally that we also have a quotient map
\[\mathcal{S}\to \mathcal{PML}\]
which has a dense image.

\subsection{Measured foliations}\label{ss:foliations}

A \emph{measured foliation} on $S$ is a foliation with isolated singularities of the type represented in Figure \ref{fig:prongs} (that is, $n$-prong singularities, where $n$ can be any integer $\geq 3$) such that each arc transverse to the foliation is equipped with a measure equivalent to a Lebesgue measure of a segment of the real line and such that these transverse measures on the various arcs are invariant by the local holonomy maps, that is, isotopies of the transverse arcs that keep each point on the same leaf.

\begin{figure}[htbp]

\centering

\includegraphics[width=9cm]{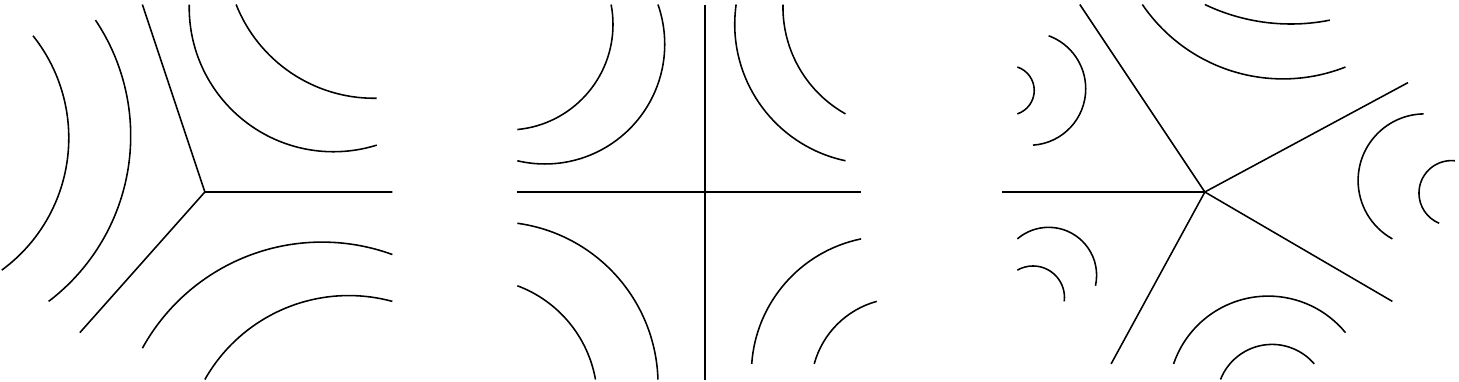}

\caption{\small{Singularities of measured foliations with 3, 4 and 5 prongs}}
\label{fig:prongs}
\end{figure}

There is an equivalence relation on the space of measured foliations  generated by isotopy and Whitehead moves. Here a \emph{Whitehead move} is an operation that consists in collapsing to a point a leaf connecting two singular points, or the inverse operation. Such an operation between measured foliations is well-defined up to isotopy. An example of a Whitehead move is given in Figure \ref{fig:Whitehead}.

\begin{figure}[htbp]
\centering
\includegraphics[width=0.9\linewidth]{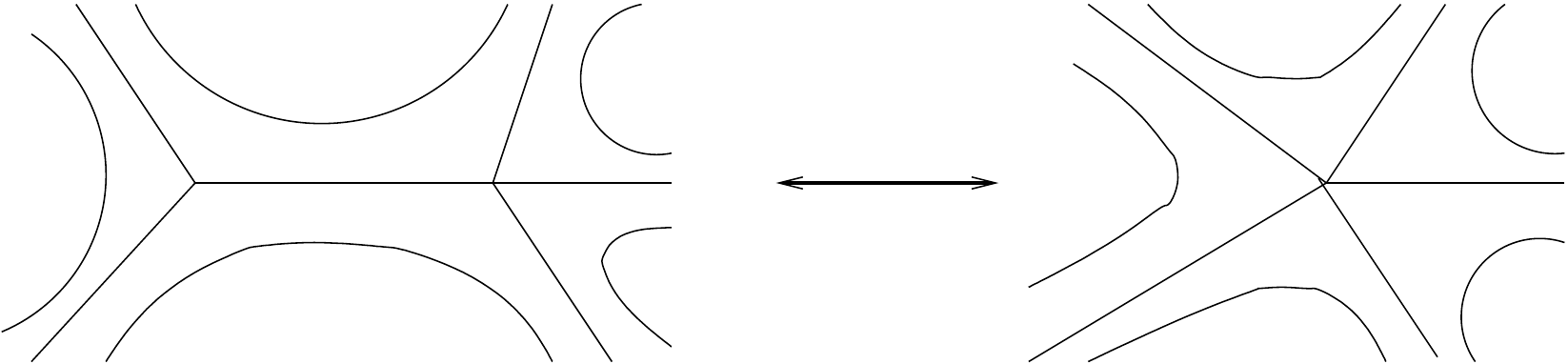}    \caption{\small {A Whitehead move}: the foliation on the right hand side is obtained from the one on the left hand side by collapsing a leaf connecting two singular points}   \label{fig:Whitehead}  
\end{figure}

A measured foliation $F$ on $S$ is said to be \emph{compactly supported} or \emph{trivial around the punctures} if each puncture of $S$ has a neighborhood homeomorphic to a cylinder on which the restriction of $F$ is a foliation by homotopic  closed leaves. We also require that the measure of any arc converging to the cusp is infinite.

The set of equivalence classes of measured foliations on $S$ is called \emph{measured foliation space} and it is denoted by $\mathcal{MF}$. All measured foliations on $S$ will be assumed to be compactly supported (trivial around the punctures). The space $\mathcal{MF}$ is equipped with a topology which we shall recall below. When the surface $S$ is equipped with a hyperbolic structure, there is a natural homeomorphism between the space $\mathcal{MF}$ of measured foliations and the space $\mathcal{ML}$ of measured geodesic laminations on $S$. The passage from a foliation to a lamination is done by replacing each leaf of a foliation by the geodesic that is homotopic to it. In the case of bi-infinite leaves, one asks that this correspondence between leaves preserves the endpoints. Making this operation precise is best seen in the universal cover of the two surfaces, both identified with the hyperbolic plane. To do this, one first shows that in the lift of the measured foliation to the universal cover,  each leaf converges in its two directions to distinct points on the boundary of the hyperbolic plane. Then,  one replaces the bi-infinite leaf with the geodesic joining the two corresponding points at infinity. This operation is done in an equivariant manner. In this way, each leaf of the measured foliation on the surface is replaced by a geodesic. The operation is called ``straightening" the foliation. The straightened foliation is a lamination. The transport of the transverse measure of the foliation to a  transverse measure of the lamination needs more work, and it involves the introduction, using the transverse measure of the foliation, of a measure on the space $\mathcal{B}= (\partial \mathbb{H}^2 \times \partial \mathbb{H}^2)\setminus \Delta$ where $\Delta$ is the diagonal of that space (that is, the subset of pairs of the form $(x,x)$ with $x$ in  $\partial \mathbb{H}^2$) and then inducing from that measure on the space $\mathcal{B}$ a transverse measure for the geodesic lamination. The measure on $\mathcal{B}$ induced by a measured foliation or a measured lamination on a surface of finite type is essentially obtained as follows: One defines the measure of each rectangular box $I\times J$, where $I$ and $J$ are arbitrary disjoint subsets of $\partial \mathbb{H}^2$, to be equal to the total transverse measure of a segment in $\mathbb{H}^2$ that is transverse to the lift of the foliation (or lamination), which intersects in a single point every leaf of this foliation which has one endpoint in $I$ and another endpoint in $J$ and which intersects no other leaves. This definition of the measures of rectangles of $\mathcal{B}$ is used to define a measure on this space. The passage between the transverse measure of a measured foliation ans the transverse measure of the associated geodesic lamination is done by using this measure induced on the space $\mathcal{B}$ of endpoints of leaves.
 We refer the reader to the paper \cite{Levitt} for the details on the passage between foliations and laminations (without, however, the discussion on the transport of the transverse measures).

 The name ``compactly supported" for a measured foliation that we introduced above is chosen because when the surface $S$ is equipped with a hyperbolic structure, the correspondence between measured foliations and measured geodesic laminations establishes a correspondence between compactly supported measured foliations and compactly supported measured geodesic laminations on $S$.

A measured foliation,  in the same way as a measured geodesic lamination, defines a function (also called intersection function) on the set $\mathcal{S}$ of homotopy classes of essential simple closed curves on $\mathcal{S}$. In fact, the intersection function associated with a measured foliation is the one defined by the measured geodesic lamination  which represents it, relative to any choice of a hyperbolic metric on $\mathcal{S}$. But this intersection function associated with a measured foliation $F$ can also be defined independently of the passage to geodesic laminations, and we recall the definition.

Let $\alpha$ be an element in $\mathcal{S}$. We represent it by a closed curve which is a concatenation of arcs that are either transverse to the leaves of $F$ or contained in leaves of this foliation. We call $\alpha'$ this representative. We define the total transverse measure of $\alpha'$ with respect to $F$, which we denote by $I(F,\alpha')$, as the sum of the transverse measures of all the subarcs of $\alpha'$ that are transverse to $F$. We then define the intersection number $i(F,\alpha)$ as the infimum of the quantities $I(F,\alpha')$ over all closed curves $\alpha'$ which are in the homotopy class $\alpha$. The collection of intersection functions $i(F,\cdot)$ associated with the various measured foliations $F$ defines an embedding of $\mathcal{MF}$ in the space $\mathbb{R}_{\geq 0}^{\mathcal{S}}$ of nonnegative functions on $\mathcal{S}$. This embedding induces a topology on the space $\mathcal{MF}$ from the weak topology on the function space  $\mathbb{R}_{\geq 0}^{\mathcal{S}}$, in the same way the embedding of the space $\mathcal{ML}$ of geodesic laminations which we described in \S \ref{ss:laminations} induces a topology on $\mathcal{ML}$.
When the spaces $\mathcal{ML}$ of (compactly supported) measured foliations on $S$ and $\mathcal{MF}$ of (compactly supported) measured geodesic laminations on $S$ (equipped with a hyperbolic structure) are endowed with their respective topologies, the passage between measured foliations and measured geodesic laminations induces a homeomorphism between these two spaces.

We pointed out at the end of \S \ref{ss:laminations} the density of the space of positively weighted simple closed curves (or, rather, the  natural image of this space) in the space  of measured geodesic laminations of $S$. Likewise, there is a natural injection of the set of positively weighted homotopy classes of simple closed curves on $S$ into the space of measured foliations of the surface:
\[\mathbb{R}_+\times \mathcal{S}\to \mathcal{MF}.\] The map is defined as follows:

 We associate to a simple closed curve $\gamma$ equipped with a positive weight $k$ a foliated annulus $A$ on $S$ whose leaves are all closed and are in the homotopy class of $\gamma$, equipped with a transverse measure which assigns to an arc transverse to the foliation and joining the two endpoints of the annulus the measure $k$. Then, we collapse each connected component of the complement in $S$ of the foliated annulus $A$ onto a graph, called a spine of the surface with boundary, obtaining a measured foliation whose equivalence class (that is, as an element $\mathcal{MF}$) does not depend on the choices made to define it. The operation is called \emph{enlarging} the simple closed curve. Figure \ref{enlarge} represents such an operation.

\begin{figure}[htbp]
\centering
\includegraphics[width=.98\linewidth]{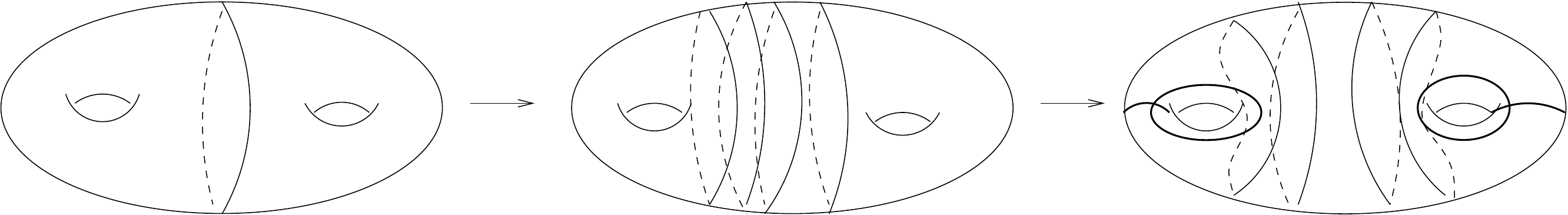}    \caption{\small {A foliation obtained by enlarging a simple closed curve. On the surface to the right, the spines of the two complementary components of the simple closed curve are represented in bold lines.}}   \label{enlarge}  
\end{figure}

The image of the map $\mathbb{R}_+\times \mathcal{S}\to \mathcal{MF}$ is dense in $\mathcal{MF}$ \cite[Corollary 4.5]{FLP}. 
The quotient map
\[\mathcal{S}\to \mathcal{PMF}\]
has also dense image.
The passage between measured foliations and measured geodesic laminations induces the identity on the natural images of the set $\mathbb{R}_+\times \mathcal{S}$ into these two spaces.

References for measured foliations include, besides Thurston's original notes \cite{Thurston-Princeton} and \cite{Thurston-FLP}, the books \cite{FLP} and \cite{PH}.

\subsection{Quasi-transverse curves}\label{ss:quasi-transverse} Given a hyperbolic structure on $S$ and a measured geodesic lamination $\mu$ on this surface, every isotopy class of essential simple closed curves $\alpha$ has a canonical representative that realizes the minimum of the intersection function with $\mu$, namely, the geodesic in the homotopy class $\alpha$. If instead of a measured geodesic lamination $\mu$ we take a measured foliation on $S$, then there is a collection of closed curves representing the isotopy class $\alpha$ that realize the minimum of the transverse measure of a curve in the class $\alpha$ with respect to $\mu$.  These representatives are the so-called quasi-transverse curves, whose theory is developed in \cite[expos\'e 5, \S 3]{FLP} and which we review now.

Let $F$ be a measured foliation on  $S$ and let $\alpha$ be again and element of $\mathcal{S}$. A closed curve $\alpha'$ in the homotopy  class $\alpha$ is said to be \emph{quasi-transverse} to $F$ if $\alpha'$ is made of a concatenation of segments that are either contained in leaves of $F$ and joining singular points, or transverse to $F$, with the additional condition that if there are two  consecutive segments of $\alpha'$ in this decomposition that meet at a singular point of $F$, among which at least one segment is transverse to $F$,  then, in the neighborhood of this singular point, the two segments are not contained in the same sector. In Figure \ref{fig:sector}, we have represented two non-allowed configurations. In Figure \ref{fig:quasi} we give an example of a piece of a quasi-transverse curve.

\begin{figure}[htbp]
\centering
\includegraphics[width=0.7\linewidth]{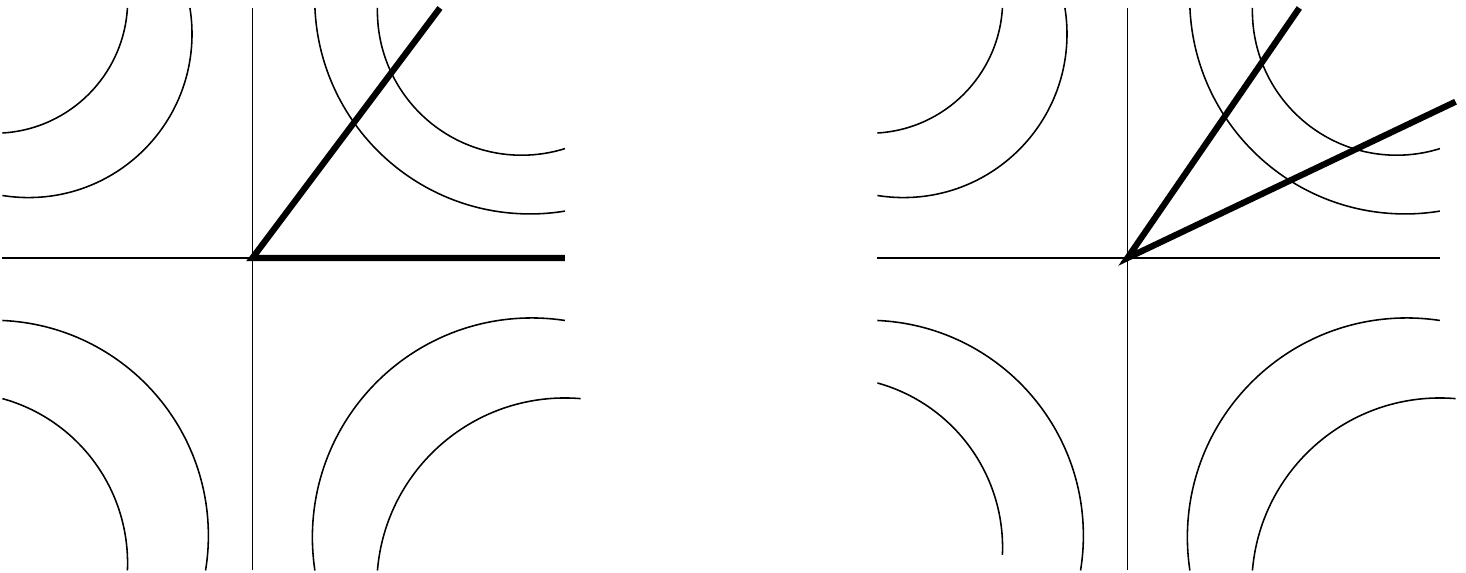}    \caption{\small {The two non-allowed configurations for a quasi-transverse curve}}   \label{fig:sector}  
\end{figure}

A quasi-transverse curve $\alpha'$ is not necessarily simple (it may have self-intersection at singular points), but it is the limit of simple closed curves.

\begin{figure}[htbp]
\centering
\includegraphics[width=0.7\linewidth]{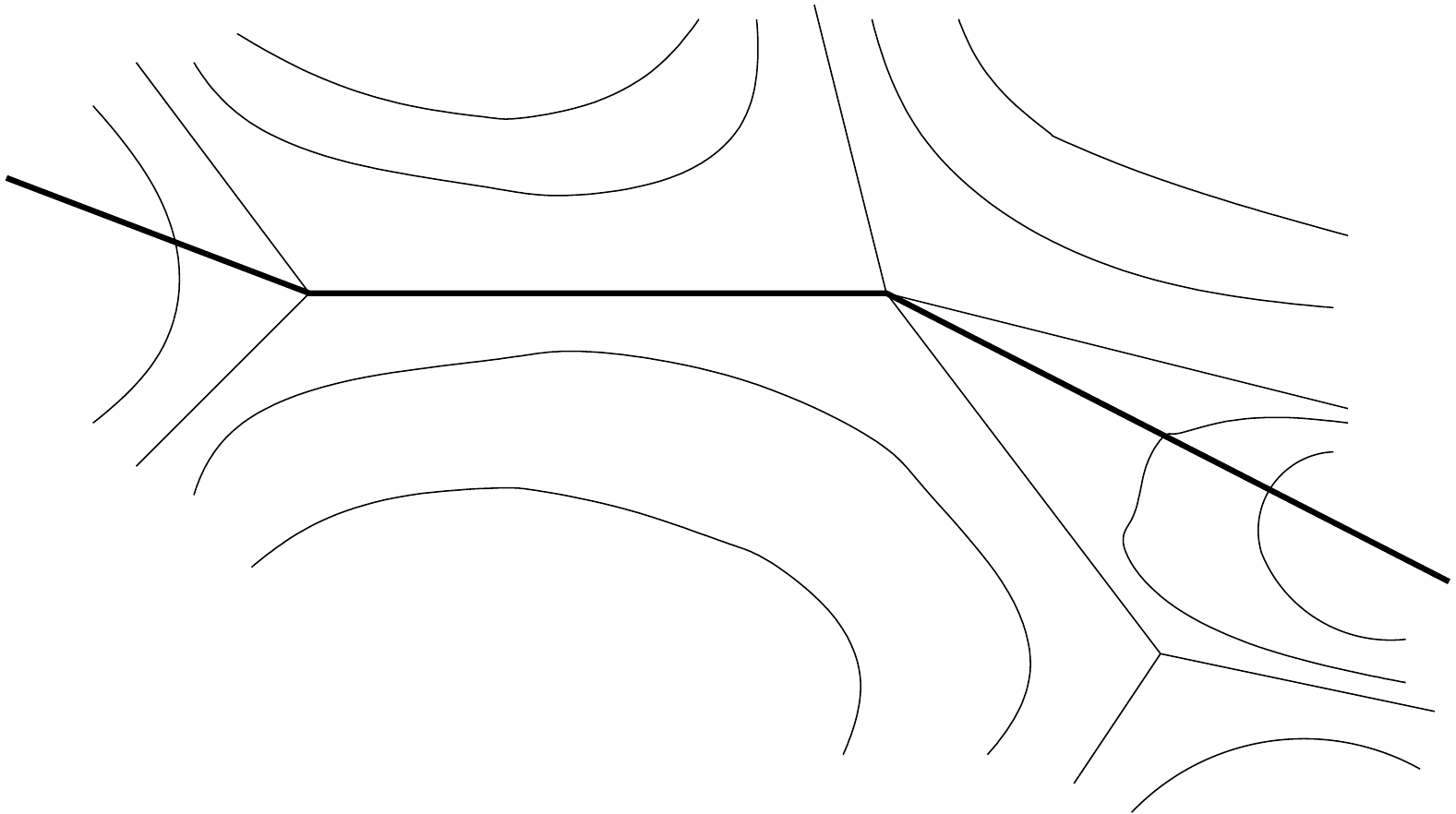}    \caption{\small {An example of a quasi-transverse curve}}   \label{fig:quasi}  
\end{figure}

The following is an important property of quasi-transverse curves (see \cite[Expos\'e 5]{FLP} for the first two parts of the statement):
\begin{proposition}
For any measured foliation $F$ and for any homotopy class of simple closed curves $\alpha$, there exists a closed curve $\alpha'$ which is quasi-transverse to $F$ and which is in the class $\alpha$. Such a representative realizes the infimum of the total intersection function with $F$, that is, $i(\alpha,F)=I(\alpha',F)$. Furthermore, any two such quasi-transverse curves that represent $\alpha$ bound a cylinder immersed in $S$ such that they can be obtained from one another by flowing along the foliation, that is, by a homotopy which leaves every point of the curve on the same leaf of the foliation induced by $F$ on that cylinder. 
\end{proposition}

The last part of the statement is in some sense a uniqueness property for a quasi-transverse representative of the element $\alpha$ in $\mathcal{S}$, in much the same way as a geodesic representative of $\alpha$ is a canonical representative of it relative to a hyperbolic metric on $S$.
  
\subsection{Thurston's compactification of Teichm\"uller space} 
There are natural actions of the multiplicative group of positive reals on the spaces $\mathcal{ML}$ and  $\mathcal{MF}$ respectively, namely, by multiplication of the transverse measures of a  lamination or foliation by a constant. The quotient spaces by these actions are denoted respectively by $\mathcal{PML}$ and  $\mathcal{PMF}$ and are called projective measured lamination space and projective measured foliation space respectively. A theorem of Thurston (see \cite[Chapter 8]{FLP} where this theorem is proved in the case of measured foliations) says that the image of each of these spaces in the projective function space $P\mathbb{R}_{\geq 0}^{\mathcal{S}}$ (see \S \ref{ss:laminations} and \S \ref{ss:foliations} above) is disjoint from the image of the embedding of Teichm\"uller space in that space defined using the geodesic length functions (see \S \ref{ss:T}). The union of these images, which we naturally call $\mathcal{T}\cup\mathcal{PML}$ and  $\mathcal{T}\cup\mathcal{PMF}$ respectively, are two equivalent versions of Thurston's compactifications of Teichm\"uller space. 

We shall use the following convergence criterion, for a sequence of points in Teichm\"uller space converging to an element of $\mathcal{PML}$ or $\mathcal{PMF}$ respectively (see \cite[Expos\'e 8]{FLP}):

\begin{proposition} Let $g_n$ be a sequence of points in Teichm\"uller space which tends to infinity in the sense that for any compact subset $K$ of $\mathcal{T}(S)$, $g_n$ is in the complement of $K$ for all $n$ large enough. Then, $g_n$ converges to an element $[\lambda]$ of Thurston's boundary $\mathcal{PML}$ or $\mathcal{PMF}$ if and only if there exists a representative  $\lambda$ of $[\lambda]$ in $\mathcal{ML}$ or $\mathcal{MF}$ respectively, and a sequence of real numbers $x_n$ ($n\geq 0$) satisfying $\lim_{n\to\infty} x_n=0$ such that for any $\alpha\in \mathcal{S}$, we have
 $x_nl_{g_n}(\alpha)\to i(\lambda,\alpha)$
as $n\to\infty$.
\end{proposition}

\section{Surfaces obtained by gluing ideal triangles} \label{s:ideal}
 The goal of this section is to show how any hyperbolic surface can be 
decomposed into a union of ideal triangles, and how the Teichm\"uller space of this surface can be parametrized by using shift parameters on the edges of an ideal triangulation.

\subsection{Horocyclic measured foliation}

An ideal triangle is equipped with a canonical partial measured foliation (that is, a foliation whose support is a subsurface) called the \emph{horocyclic measured foliation}. The leaves of this foliation are pieces of horocycles that are centered at the three vertices of the triangle. The non-foliated region of the triangle is a central region bounded by three pieces of horocycles, meeting each other  tangentially. To see that this foliation is  determined by the last property, we may consider first the case of an ideal triangle which is symmetric with respect to the Euclidean metric in the disc model of hyperbolic space (see Figure  \ref{fig:horo}). The (Euclidean) symmetry of that disc shows that this foliation is unique. We then appeal to the fact that any two ideal triangles are isometric.

The horocyclic foliation carries a canonical transverse measure which is uniquely determined by the property that on the edges of the ideal triangle, this measure coincides with hyperbolic distance.  
\begin{figure}[htbp]

\centering

\includegraphics[width=12cm]{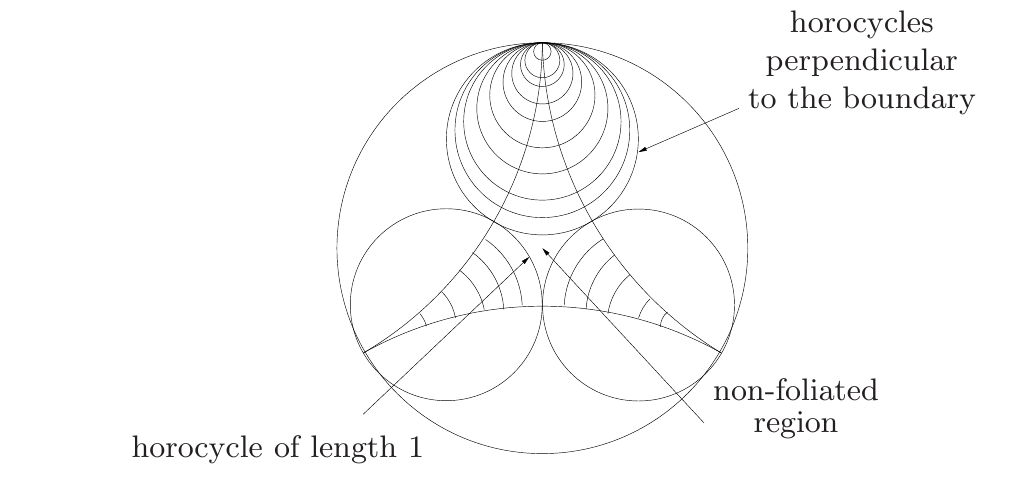}

\caption{\small The horocyclic foliation of an ideal triangle}
\label{fig:horo}
\end{figure}

We may replace the measured horocyclic partial foliation of the idea triangle by a measured foliation of full support with a 3-prong singularity, as indicated in Figure \ref{fig:collapse}, and obtain a measured foliation on the ideal triangle which is well-defined up to an isotopy of the ideal triangle which preserves the boundary pointwise.

\begin{figure}[htbp]

\centering

\includegraphics[width=12cm]{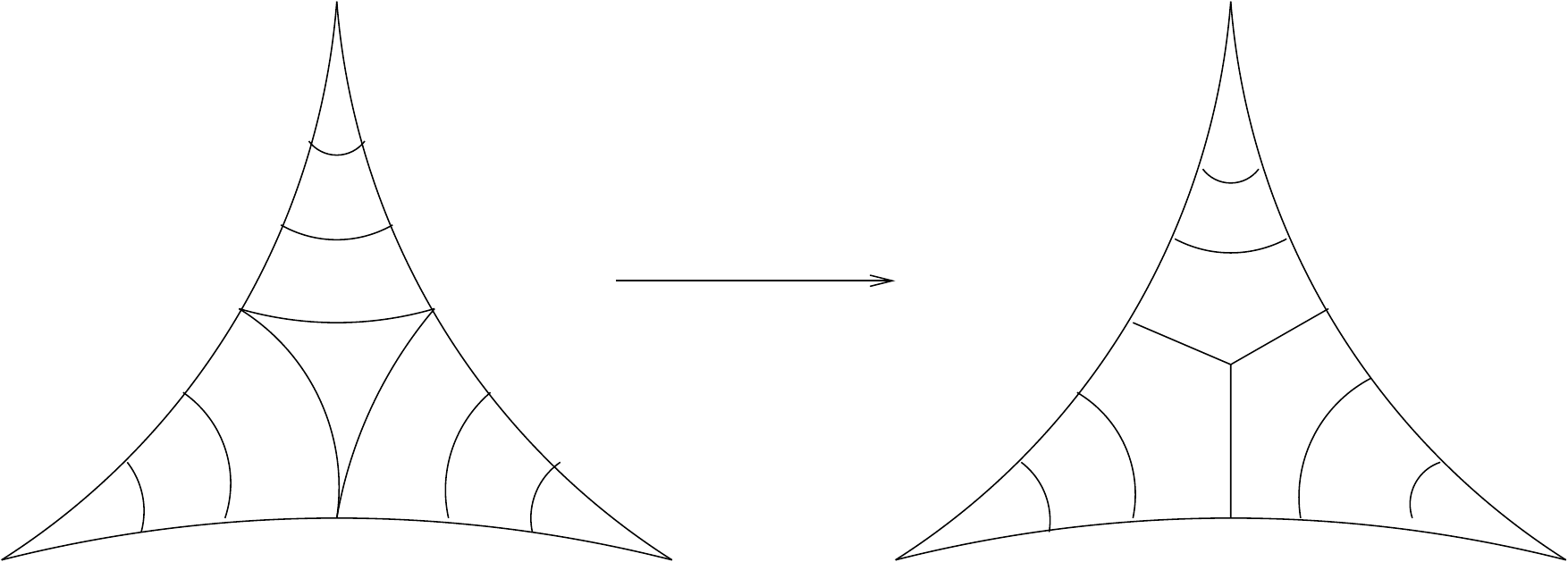}

\caption{\small Collapsing the non-foliated region of the horocyclic measured foliation of an ideal triangle onto a tripod}
\label{fig:collapse}
\end{figure}

From horocyclic foliations of ideal triangles, we pass now to horocyclic foliations associated with maximal geodesic laminations.

The complement  $S\setminus \mu$ of a geodesic lamination $\mu$ on a hyperbolic surface $S$ consists of finitely many subsurfaces whose completions are subsurfaces with boundary. 

A geodesic lamination $\mu$ is said to be \emph{maximal} if each such completion is isometric to an ideal triangle.

 Let $g$ be a hyperbolic structure on $S$ and let $\mu$ be a maximal geodesic lamination on the hyperbolic surface $(S,g)$.  We equip each ideal triangle in the complement of $\mu$ with its horocyclic measured foliation. These measured foliations match together nicely: their leaves are perpendicular to the boundary of the ideal triangle, and when two triangles share a common edge,  the transverse measures on that edge from each side coincide. The union of these foliations of ideal triangles is a measured foliation on $S$ called the \emph{horocyclic measured foliation} of $\mu$ with respect to $g$. The equivalence class of this  measured foliation is well defined as an element of $\mathcal{MF}$. We denote it by $F_{\mu}(g)$. With the definition we gave, the fact that the hyperbolic metric $g$ is complete is equivalent to the fact that the associated horocyclic foliation $F_{\mu}(g)$ is compactly supported.

\subsection{Gluing two ideal triangles}\label{gluing} We shall study the geometry of a hyperbolic surface homeomorphic to a pair of pants (3-punctured sphere) which is obtained by gluing two ideal triangles along their boundary. Before this, we observe that there are two combinatorially distinct ways of gluing two ideal triangles along their boundaries to obtain, from the topological point of view, a pair of pants. These two combinatorial gluings are described schematically in Figure \ref{fig:two}.  The topological results of these gluings, seen as decompositions of the 3-punctured sphere  into two triangles, are represented in Figure \ref{fig:combinatorial}, where the three punctures correspond to the three vertices $A,B,C$ of the triangulations drawn on the sphere, arising from the vertices of the triangles that are glued.

\begin{figure}[htbp]
\centering
\includegraphics[width=7cm]{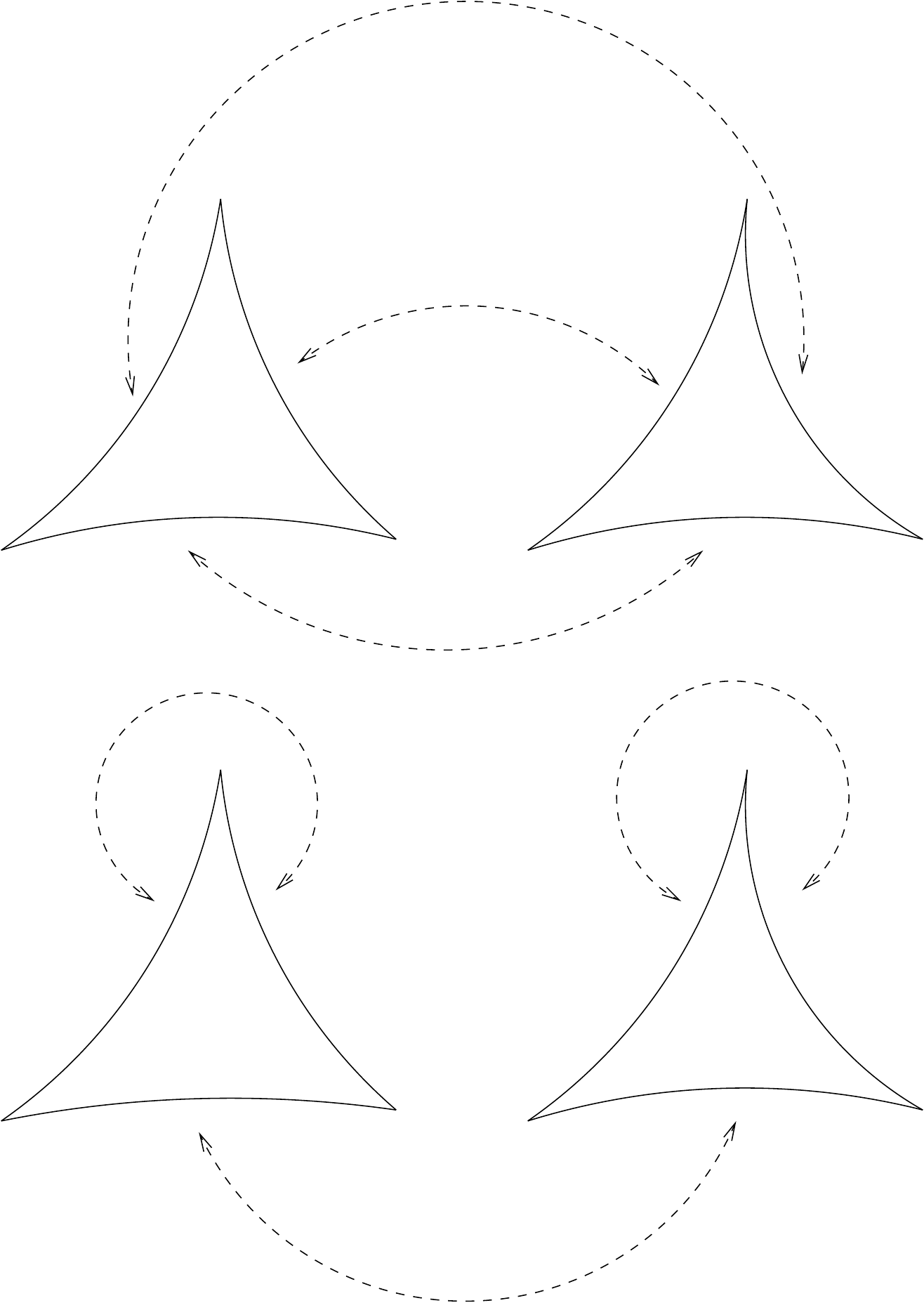}
\caption{\small Two combinatorially different gluings of two ideal triangles that give a hyperbolic pair of pants}
\label{fig:two}
\end{figure}

\begin{figure}[htbp]
\centering
\includegraphics[width=10cm]{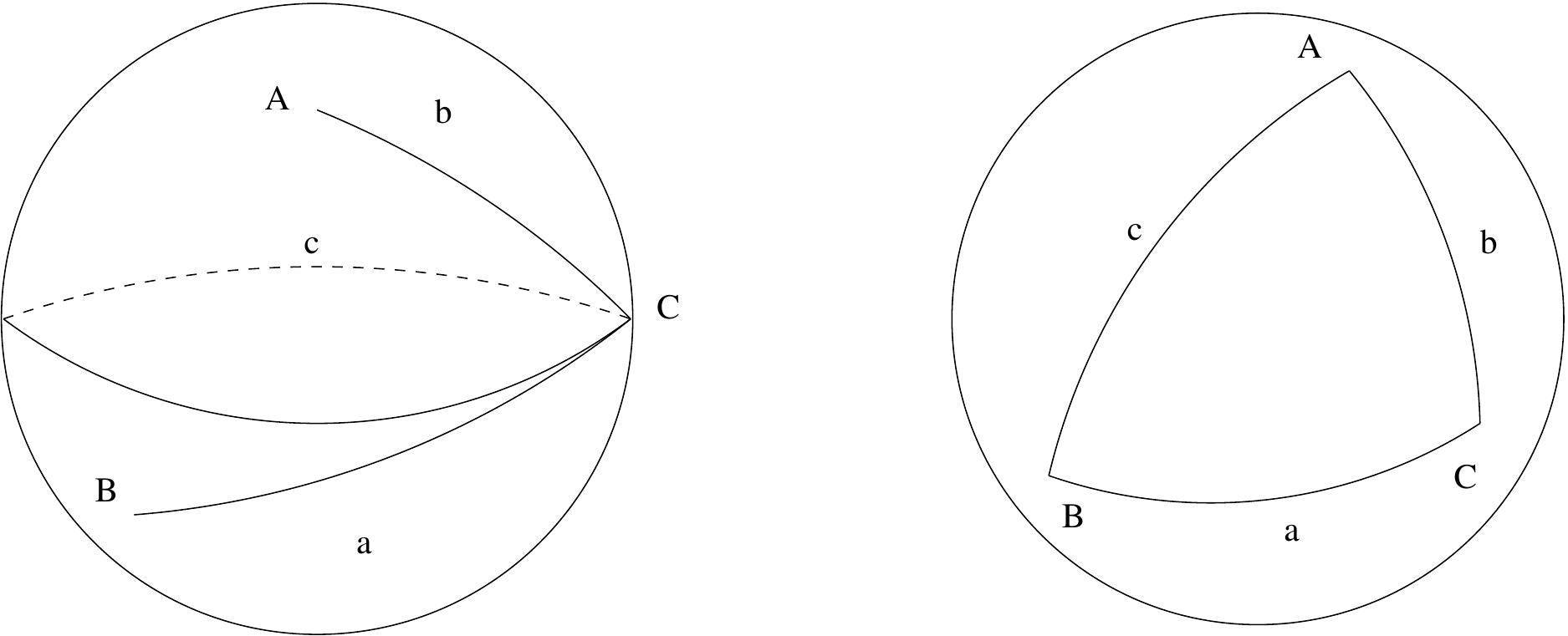}
\caption{\small The two combinatorial triangulations of the sphere obtained by the two gluings of Figure \ref{fig:two}. In each case, the graph drawn is the 1-skeleton of the triangulation. In the figure on the left hand side, when we cut the sphere along the 1-skeleton of the triangulation, we find two triangles, one with side $c$ and two sides called $b$ which are glued together, and the other one with side $c$ and two sides called $a$ which are glued together. On the right hand side, the two triangles that are glued have vertices $A,B,C$ and sides $a,b,c$.}
\label{fig:combinatorial}
\end{figure}
 
From the geometric point of view, the result of each of these gluings may give different types of hyperbolic pairs of pants (see Figure \ref{fig:4cases}): they may be complete or not, they may have 0, 1, 2 or 3 cusps.  The geometric type of the pair of pants obtained is best described using Thurston's shift (or shear) parameters which we now review.
\begin{figure}[htbp]

\centering
 \psfrag{(i)}{\small $(i)$}
  \psfrag{(i)}{\small $(ii)$}
   \psfrag{(i)}{\small $(iii)$}
    \psfrag{(i)}{\small $(iv)$}
\includegraphics[width=12.3cm]{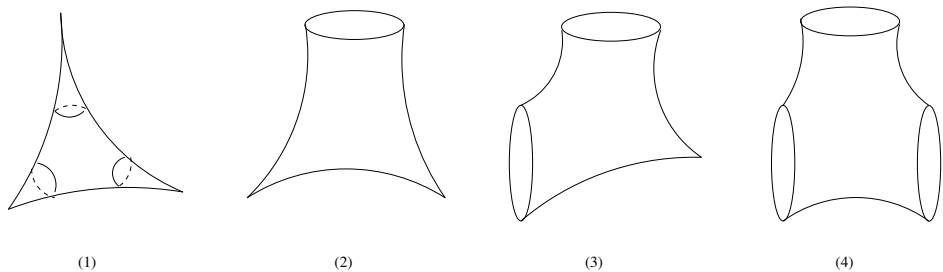}

\caption{\small The four kinds of geometric pairs of pants obtained by gluing two ideal triangles}
\label{fig:4cases}
\end{figure}

To define these coordinates, we first note that on each edge of an ideal triangle, there is a \emph{distinguished point}, namely, the intersection point of this edge with the boundary of the nonfoliated region of the horocyclic foliation of this triangle. Equivalently, this distinguished point is the intersection point of the given edge with the singular leaf of the foliation, when the nonfoliated region is collapsed onto a tripod (as in Figure \ref{fig:collapse}). Gluing two ideal triangles along an edge (the two triangles might be the same, but the edges different) induces on this edge a \emph{shift parameter}. This is the signed distance (a distance equipped with a sign) between the two distinguished points that are associated with that edge, when this edge is regarded as an edge of two ideal triangles, one from each side. (Note that the two triangles that are glued together might be the same, but the edges must be distinct.) The distance between the distinguished points is measured using the notion of hyperbolic length on that edge, and the sign is positive or negative depending on whether an observer standing on that edge and looking towards a triangle adjacent to that edge, sees that these two distinguished points differ by a left shift or a right shift respectively. (The sign does not depend on the choice of the side to which the observer looks, and it does not use any orientation on that edge; it depends only on the choice of an orientation of the surface.) The two cases lead to two different signs of the shifts and they are represented in Figure \ref{fig:positive}.

\begin{figure}[htbp]

\centering

\includegraphics[width=10cm]{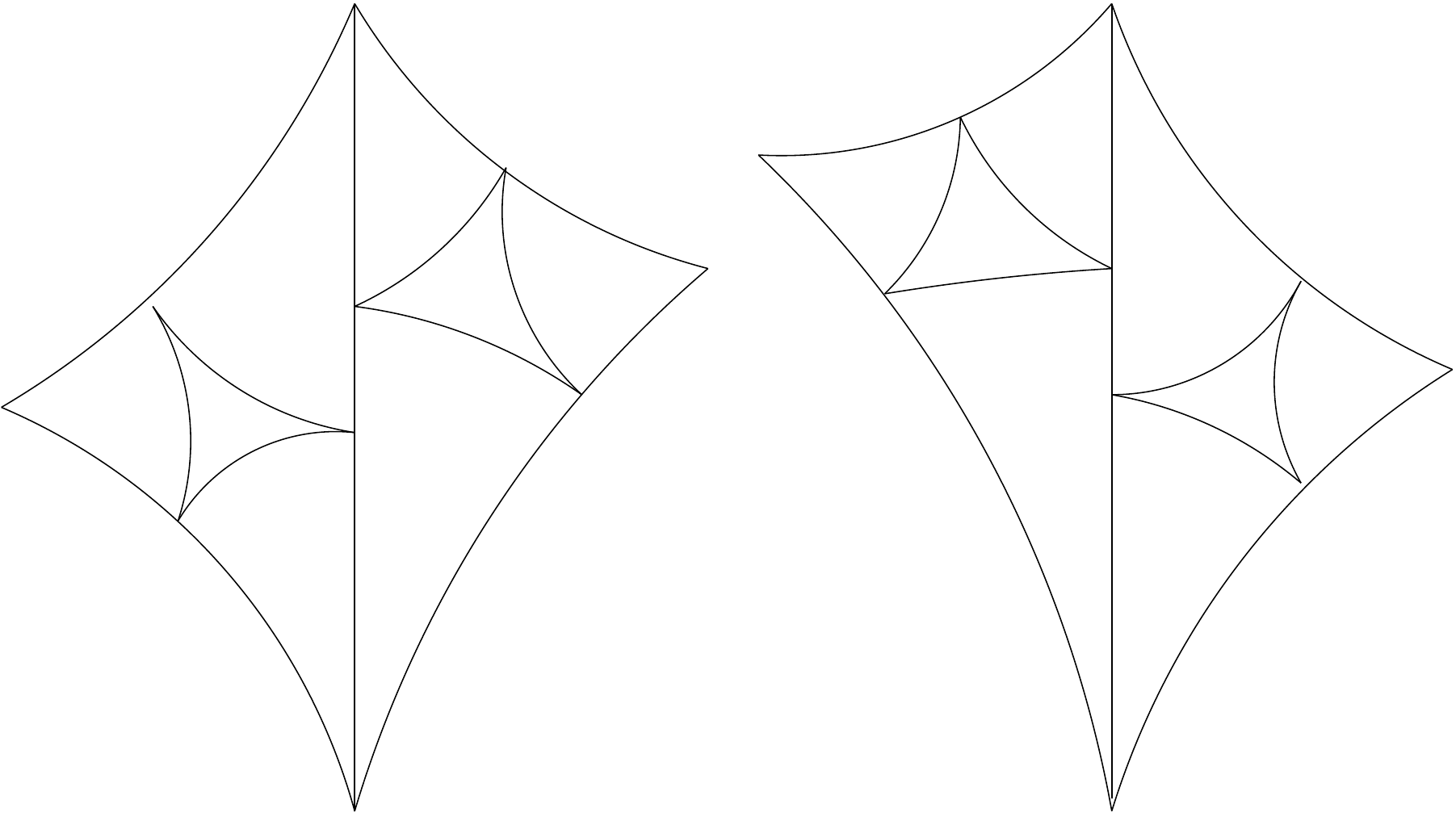}

\caption{\small Gluing two ideal triangles along an edge. The distinguished points are indicated on the edges that are glued. On the figure on the left hand side, the shift is negative and on the one on the right-hand side it is positive.}
\label{fig:positive}
\end{figure}

\begin{figure}[htbp]

\centering

\includegraphics[width=10cm]{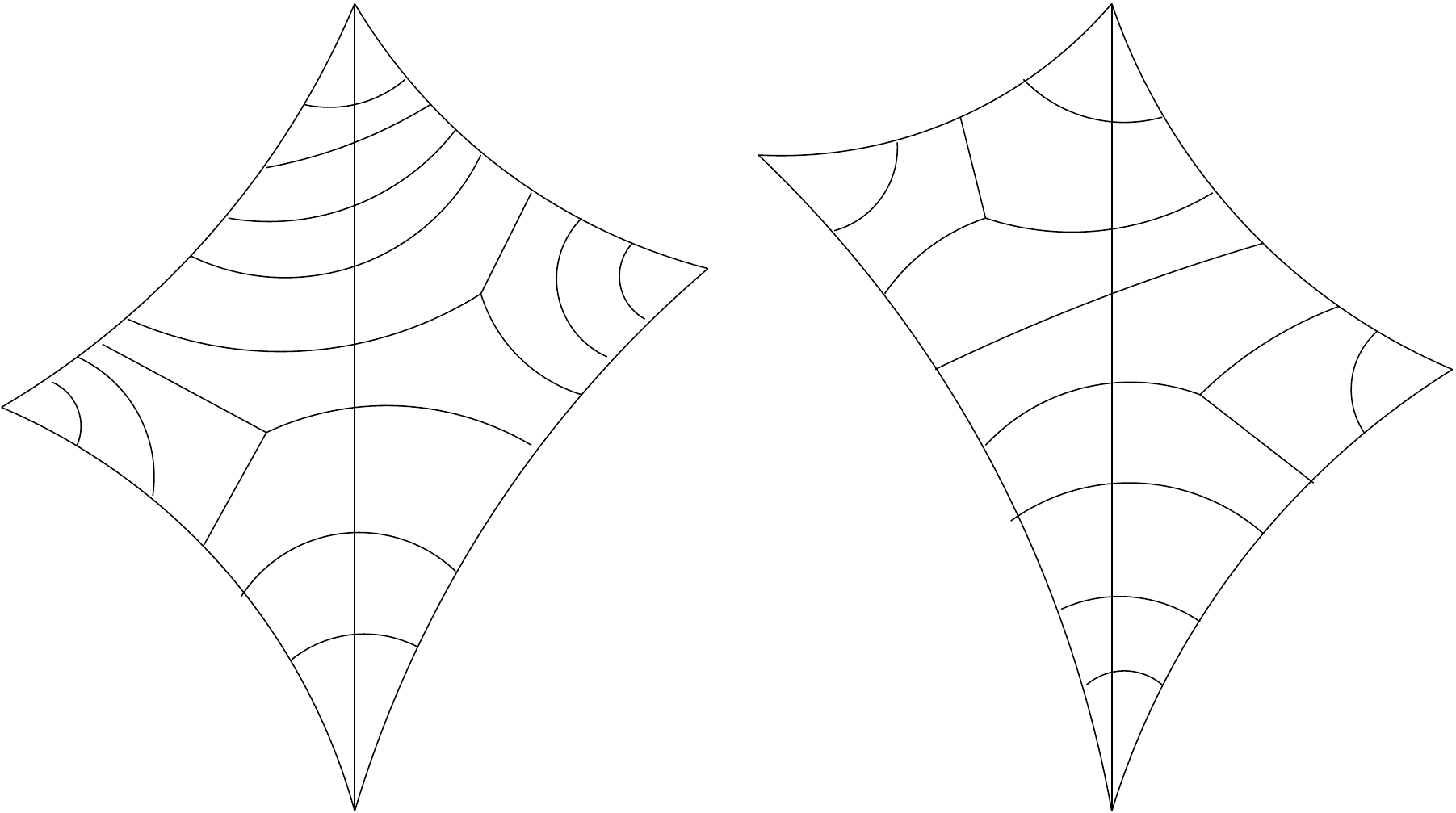}

\caption{\small Shift coordinates for measured foliations. On the left-hand side, the shift is negative and on the right-hand side it is positive.}
\label{fig:shift-foliations}
\end{figure}

Now for an arbitrary vertex of the triangulation of the sphere represented in Figure \ref{fig:combinatorial}, there are 1, 2 or 4 half-edges which locally terminate at this vertex:  On the sphere represented on the left hand side of that figure, there is one half-edge terminating at $A$ and there are three half-edges terminating at $C$. On the one represented on the right hand side, there are two half-edges terminating at each of the points $A,B,C$. The resulting geometric pair of pants obtained by gluing in either way the two ideal triangles may be complete or not (Figure \ref{fig:4cases}). We know that for this surface to be complete, the neighborhood of each puncture must be, geometrically, a cusp. This depends on the sum of the shift parameters of the half-edges that terminate at that vertex.  The precise result is the following:

\begin{proposition}\label{prop:gluing2}
Let $v$ be a puncture of a pair of pants obtained by gluing two ideal triangles.  Then, the hyperbolic structure at this puncture is complete (or, equivalently, the puncture is a cusp) if and only if the sum of  the shift parameters of the half-edges that terminate at that puncture is zero.

 In the case where the hyperbolic structure at the puncture is not complete, then the completion of the surface at that puncture is obtained by adjoining to it a simple closed curve which makes this surface at that puncture a surface with boundary. The added boundary component is geodesic. Furthermore, the length of this boundary component is equal to the absolute value of the sum of  the shift parameters of the half-edges that terminate at the given puncture.

\end{proposition}
Thus, depending on the vanishing or not of the sum of the shift parameters at each puncture, we obtain one of the pairs of pants represented in Figure \ref{fig:4cases}. The pair of pants is complete only in the case represented on the left hand side. This is the case where at each puncture the sum of the shift parameters is zero. 
 
For the proof of this proposition, we refer the reader to \cite[\S 3.4]{Thurston-Princeton} or Proposition 4.1 of \cite{Shift}. This proof involves the study of the developing map of the hyperbolic structure of the pair of pants.

\begin{figure}[htbp]

\centering

\includegraphics[width=8cm]{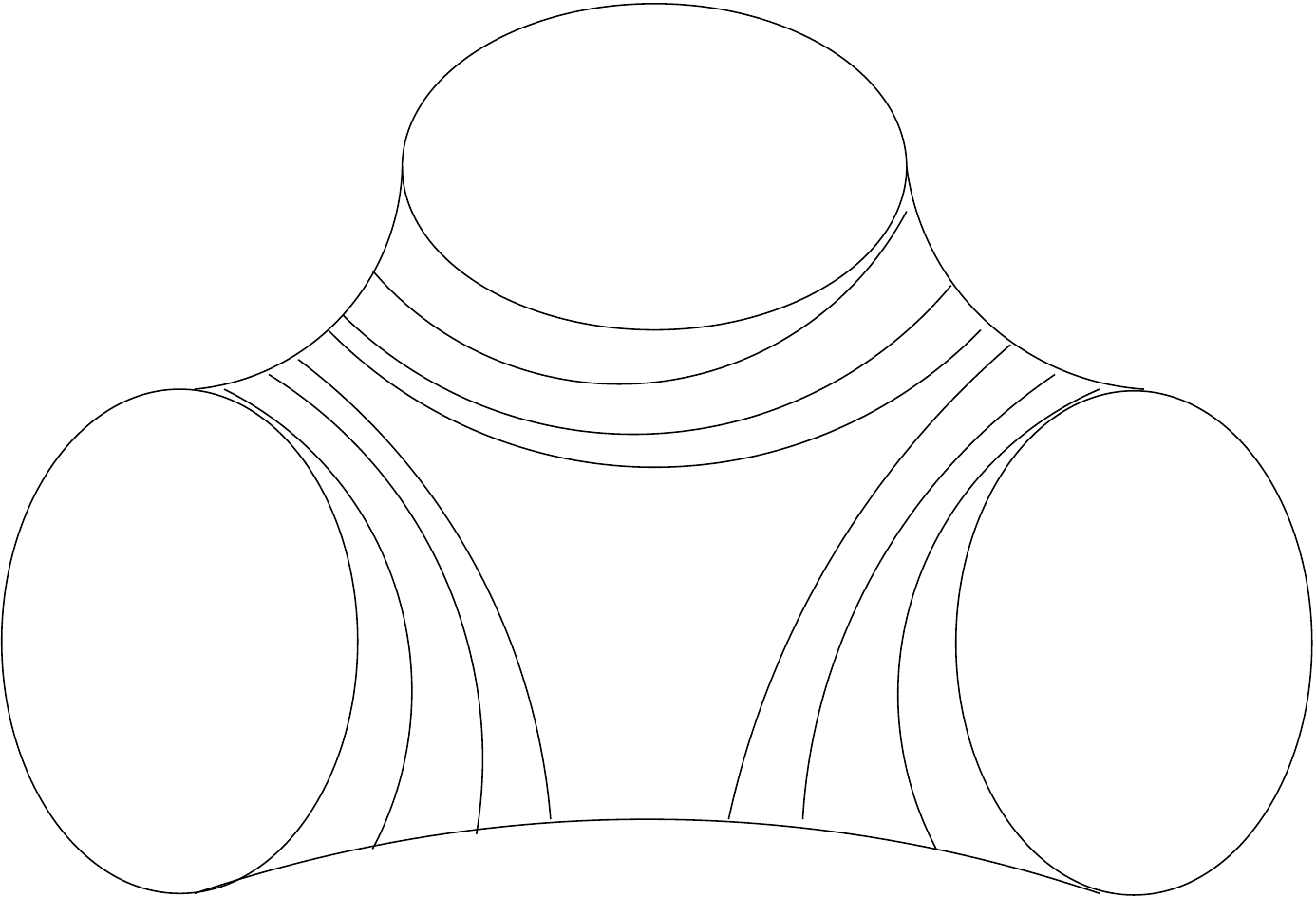}

\caption{\small A pair of pants with three boundary components obtained as a union of two ideal triangles. In the case represented, each of the two ideal triangles has one cusp spiraling along one of the boundary components of the pair of pants. In this pair of pants, the combinatorial type of the triangulation of the sphere corresponds to the one on the right hand side of Figure \ref{fig:combinatorial}.
The sum of the shifts at each puncture is nonzero, so that such a puncture becomes, when the surface is completed, a boundary component.}
\label{fig:spirals}
\end{figure}

\begin{figure}[htbp]

\centering

\includegraphics[width=8cm]{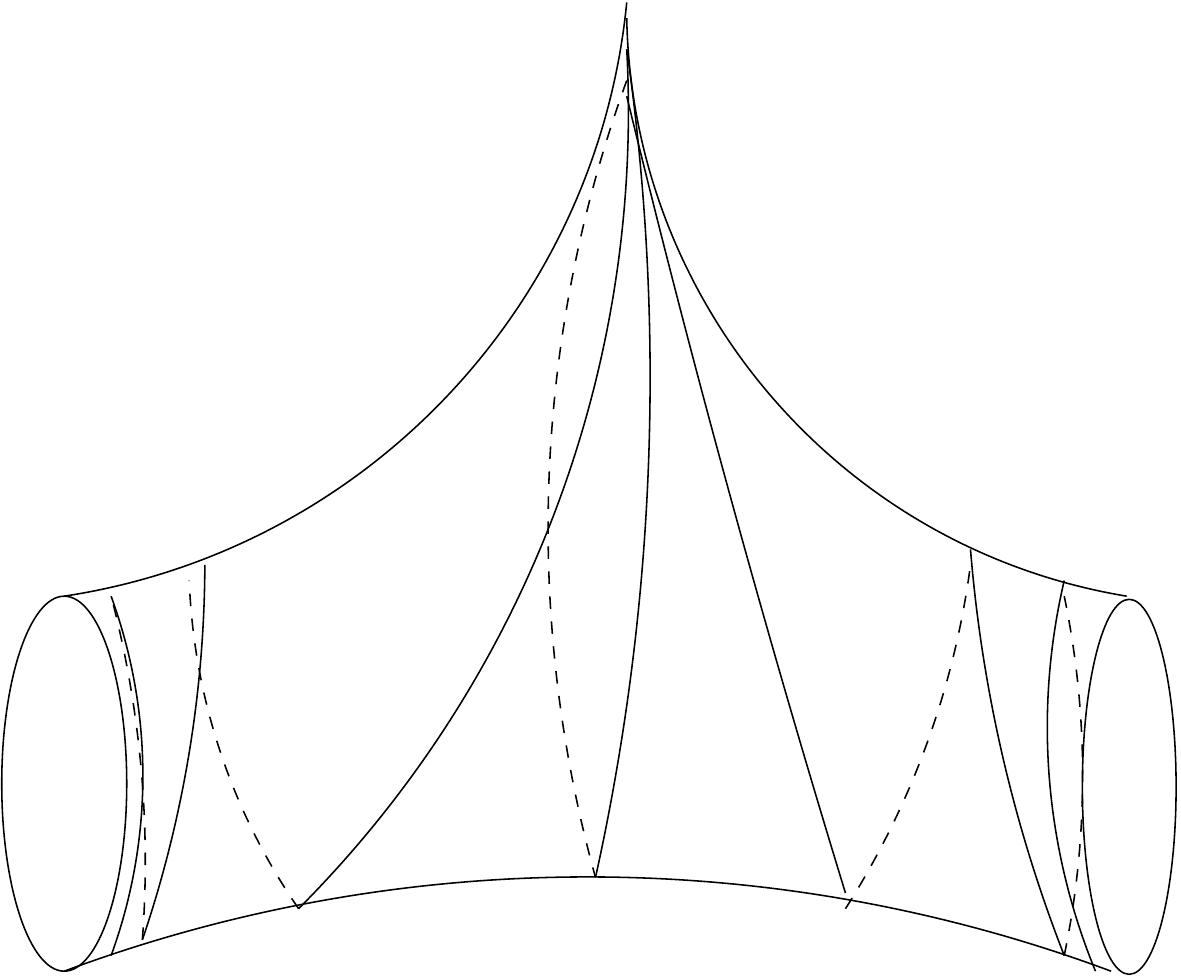}

\caption{\small  A pair of pants with two boundary components and one cusp obtained as a union of two ideal triangles. In the case represented, each of the two ideal triangles has one cusp converging to a cusp of the surface and another one spiraling along a boundary component .
 In this example, the combinatorics of the triangulation of the sphere is the one represented on the left hand side of Figure \ref{fig:combinatorial}.}
\label{fig:spirals2}
\end{figure}

\subsection{Other surfaces obtained by gluings ideal triangles} There is another way of gluing two ideal triangles, which gives a surface homeomorphic to a
  torus with one puncture, see Figure \ref{fig:torus}. A result analogous to that of Proposition \ref{prop:gluing2} holds: the gluing leads to a complete surface if and only if the sum of the shift parameters at the puncture (that is, the vertex of the combinatorial triangulation induced by the two triangles) is zero. Again, if this sum is nonzero, then the completion of the hyperbolic surface is obtained by adjoining to that surface a simple closed curve, and the surface becomes, at the given puncture, a surface with boundary, with the boundary being a simple closed geodesic. Furthermore, the length of this boundary component is equal to the absolute value of the sum of  the shift parameters of the half-edges that terminate at that puncture.

\begin{figure}[htbp]

\centering

\includegraphics[width=7cm]{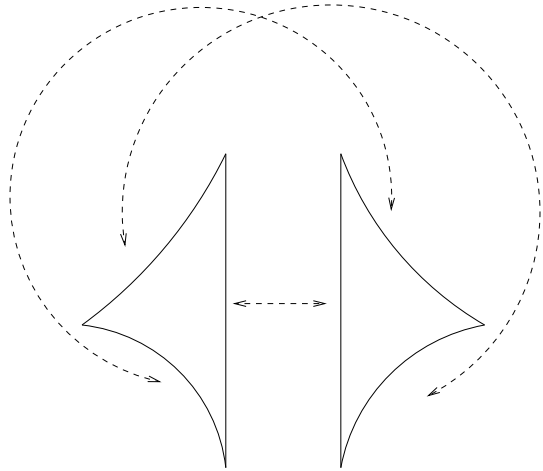}
 \psfrag{A}{\small $A$}
  \psfrag{B}{\small $B$}
 \psfrag{C}{\small $C$}
\psfrag{a}{\small $a$}
  \psfrag{b}{\small $b$}
 \psfrag{c}{\small $c$}
\caption{\small A gluing of two ideal triangles that gives a torus with one hole (one must be careful about making the orientations match)}
\label{fig:torus}
\end{figure}

More generally, a result analogous to that of Proposition \ref{prop:gluing2} holds for the gluing of any finite set of ideal triangles which gives a hyperbolic surface:  the surface at an arbitrary puncture is complete if and only if the sum of the half-edges terminating at that puncture is zero. In the case where the surface is not complete, its completion is obtained by adjoining to the surface, at each puncture, a closed geodesic whose length is equal to the absolute value of sum of the shifts of the half-edges that terminate at that puncture.

For the purpose of constructing closed surfaces of any arbitrary genus, it suffices to consider the pairs of pants with geodesic boundaries that we described above, which are obtained by gluing two ideal triangles, and to glue together such pairs of pants that have the same boundary lengths.  
 
\section{Thurston's metric}\label{s:metric}
In this section, we introduce Thurston's metric on Teichm\"uller space and we review some of its properties. As in the previous sections, $S$ is an oriented surface of finite topological type with negative Euler characteristic.

\subsection{The two definitions of Thurston's metric}\label{ss:Thurston}

Let $g$ and $h$ be two hyperbolic structures on  $S$ and let $\varphi:(S,g)\to (S,h)$ be a 
diffeomorphism which is homotopic to the identity.  
The \emph{Lipschitz constant} $\hbox{Lip}(\varphi)$ of $\varphi$ 
is defined by
\begin{displaymath}
\hbox{Lip}(\varphi)=\sup_{x\neq y\in S}\frac{d_{h}\big{(}\varphi(x),\varphi(y)\big{)}}{d_{g}\big{(}x,y\big{)}}.
\end{displaymath}
The infimum of such Lipschitz constants over all diffeomorphisms
  $\varphi$ in the isotopy class of the identity map of $S$ is denoted by
\begin{displaymath}
L(g,h)=\log\inf_{\varphi\sim\mathrm{Id}_{S}}\,\hbox{Lip}(\varphi).
\end{displaymath}

It is obvious that $L$ satisfies the triangle inequality. Furthermore, it satisfies $L(g,h)\geq 0$ for all $g$ and $h$ and $L(g,h)= 0$ if and only if $g=h$; see Thurston's proof of this result in \cite[Proposition 2.1]{Thurston1986}, based on the fact that any two hyperbolic structures on $S$ have the same area.

Varying $g$ and $h$ in their respective homotopy classes does not change the value of $L(g,h)$. Thus, $L$ may be considered as a function on $\mathcal{T}(S)\times  \mathcal{T}(S)$. We shall denote this new function by the same letter:
\begin{displaymath}
L: \mathcal{T}(S)\times \mathcal{T}(S)\to [0,\infty).
\end{displaymath}

This function is an asymmetric metric. In other words, it satisfies all the axioms of a metric except the symmetry axiom. Indeed, using a classical estimate on the width of a collar around a short geodesic, it is easy to see that there are elements $g$ and $h$ in $\mathcal{T}(S)$ satisfying $L(g,h)\not=L(h,g)$; see Thurston \cite[p. p. 12]{Thurston1986}. 
  One may construct metrics $g$ and $h$ with explicit formulae for the distances $d(g,h)\not= d(h,g)$; see e.g. Theorem \ref{th:cylindrical-anti-stretch} below, due to Th\'eret.  

 The distance function $L$ is called \emph{Thurston's metric} on Teichm\"uller space;
 
Thurston's development of this metric theory is based on the construction of some Lipschitz maps between ideal triangles, which lead to a construction of geodesics for this metric, called  stretch lines. We shall review stretch lines in \S \ref{s:stretch} below. They are one-parameter families of stretch maps between Riemann surfaces. Stretch maps are obtained by combining Lipschitz  maps between ideal triangles so as to get Lipschitz maps between hyperbolic surfaces decomposed into unions of ideal triangles.   
  
 Thurston showed that the distance $L(g,h)$ between two hyperbolic surfaces can also be computed by comparing lengths of closed geodesics in corresponding homotopy classes,
measured with the metrics $g$ and $h$. We recall the definition: 

If $h$ is a hyperbolic metric on $S$ and $\gamma$ an element of $\mathcal{S}$, we recall that $l_h(\gamma)$ denotes the length of the unique geodesic (for the metric $h$) in the homotopy class   $\gamma$. Then, for any $\gamma$ and for any two hyperbolic metrics $g$ and $h$ on $S$, we  
     set
\[
K(g,h)=\log\sup_{\gamma\in\mathcal{S}}\,\frac{l_{h}(\gamma)}{l_{g}(\gamma)}.
\]

It is immediate to see that the function $K$, like the function $L$, does not change if the hyperbolic metrics $g$ and $h$ vary in their homotopy classes. Thus, $K$ defines a function on the Teichm\"uller space $\mathcal{T}(S)$,  which we denote by the same letter. It is clear from the definition that $K$ satisfies the triangle inequality.  We have $K(g,h)>0$ for all
 $g\not=h$. This is proved by Thurston in \cite[Theorem 3.1]{Thurston1986}.

Since the length of a curve under a $k$-Lipschitz map between surfaces is multiplied by a factor which is $\leq k$, it is easy to see that
 $K\leq L$. Thurston proved the following:
 
\begin{theorem}[Thurston  \cite{Thurston1986}, Theorem 8.5]\label{K=L}
L=K.
\end{theorem}

The two definitions of the Thurston metric are reminiscent of two definitions of the Teichm\"uller metric, which we recall briefly. Here, the Teichm\"uller space of the surface is considered as the space of isotopy classes of conformal structures on the base surface $S$. The equivalence of this definition with the definition that we used before in this paper follows from the uniformization theorem, which assigns to each conformal structure a (canonical) complete hyperbolic structure whose underlying conformal structure is the given one.

We give a first definition of the Teichm\"uller distance on Teichm\"uller space which is analogous to the $L$ version of the Thurston metric.

Given two conformal structures $G$ and $H$ on the surface $S$, we define their distance by the formula
\[\frac{1}{2}\inf_f D(f)
\] 
where the infimum is taken over all quasiconformal homeomorphisms $f : G \to H$
that are in the isotopy class of the identity and where $D(f)$ is the quasiconformal dilatation of $f$.
Making the two conformal structures $G$ and $H$ vary in their hyperbolic classes (that is, considering them as elements of the Teichm\"uller space) does not change the value of this distance between them and leads to a metric on Teichm\"uller space, which is the Teichm\"uller metric. The definition of this metric is due to Teichm\"uller \cite{Teichmuller}.

 The second formula for the Teichm\"uller distance is due to Kerckhoff \cite{Kerckhoff}, and it uses the notion of extremal length of a homotopy class of simple closed curves on a Riemann surface. It is the analogue of the $K$ version of the Teichm\"uller metric. Kerckhoff's formula says that the distance between two conformal structures $G$ and $H$ on $S$ is given by 
 \[\frac{1}{2}\log\sup_{\gamma}\frac{\mathrm{Ext}_H(\gamma)}{\mathrm{Ext}_G(\gamma)}
 \]
 where the sup is taken over the set of homotopy classes $\gamma$  of simple closed curves on $S$ and where for each $\gamma$, $\mathrm{Ext}_H(\gamma)$ denotes its extremal length with respect to the conformal structure $H$.

 Despite the formal analogy between the definitions, the  techniques that are used in the development of the theories that they lead to are different. For the Thurston metric, the techniques are purely geometric (distance geometry in the hyperbolic plane, properties of horocycles, of hypercycles, hyperbolic trigonometry, etc.) whereas for the Teichmüller metric, they are mainly analytical. Furthermore, Teichm\"uller's distance is symmetric whereas Thurston's distance is not. In the next sections, we shall point out at a few places some comparison between results that hold for these metrics.

  Theorem \ref{K=L} is one of the major results of the paper \cite{Thurston1986}. Its proof involves in a strong way the existence of stretch lines, which are geodesics for the metric $K$ (or $L$) which we shall review below and which are used at many places in the development of this theory.

There is a useful alternative version of the definition of the metric $K(g,h)$ in which the supremum is taken over the space $\mathcal{PML}$ of projective geodesic laminations. Before presenting it, we make a digression on the notion of length of measured geodesic laminations. 

The notion of length of a  weighted simple closed geodesic can be extended in a natural way to a notion of length of a measured geodesic lamination. One way of defining the length of a geodesic laminatio is the following. 

Consider a measured geodesic lamination $\mu$ on a hyperbolic surface. We start by  covering its support by a family of rectangles (that is, topological discs embedded in the surface with four distinguished points on their boundaries) $R_1,\ldots, R_k$ with disjoint interiors, such that each leaf of $\mu$ intersects a rectangle in geodesic segments which all join the same pair of opposite sides of this rectangle. We call these sides the \emph{vertical} sides of the given rectangle. In this way, each time a leaf of $\mu$ enters some rectangle $R_i$ ($0\leq i\leq k$) from a vertical side, it exits the same rectangle from the opposite vertical side. 

The length of the measured geodesic lamination $\mu$ is then defined as the sum over the set of all the rectangles that cover $\mu$, of the \emph{mass of $\mu$} in such a rectangle. Here, the mass of $\mu$ in a rectangle is the integral over a vertical side of that rectangle (with respect of the transverse measure induced by the geodesic lamination on that segment), of the function which assigns to each point the length of the segment in which the lamination $\mu$ traverses the rectangle, passing through this point.

From the invariance of the transverse measure of a lamination, it follows that the length of a measured geodesic lamination, defined in this way, does not depend on the choice of the set of rectangles that cover the support of $\mu$. 

The length of a measured geodesic lamination with respect to a hyperbolic structure $g$ is denoted by $l_g(\mu)$.

We reviewed in \S \ref{ss:laminations} the canonical injection of the set $\mathbb{R}_+\times \mathcal{S}$ of weighted simple closed geodesics in the space $\mathcal{ML}$ of measured geodesic laminations.  With the definition of length of a measured lamination that we just recalled, when $\mu$ is a simple closed geodesic, $l_g(\mu)$ is the usual length of the curve $\mu$ with respect to the hyperbolic metric $g$. We shall une the following homogeneity property that follows easily from the definitions:
\begin{proposition}\label{homogeneity}
If $k\mu$ denotes a simple closed geodesic weighted by a positive real number $k$, then $l_g(k\mu)= kl_g(\mu)$. 
\end{proposition}

Proposition \ref{prop:extension} says that the set of weighted simple closed geodesic is dense in the space $\mathcal{ML}(S)$ of measured geodesic laminations equipped with the measure topology. The notion of length function of a measured lamination is the unique continuous extension of the notion of length function of a weighted simple closed geodesic. 

Let us note finally the following characterization of the length of a maximal measured geodesic lamination, in terms of its intersection number with its horocyclic measured foliation:
\begin{proposition} For any maximal measured geodesic lamination $\mu$ on a hyperbolic surface $g$, we have 
$l_g(\mu)=i(\mu, F_{\mu}(g))$.
\end{proposition}

 The proposition is proved in \cite[Lemma 3.13]{1991a}. It is generalized in \cite[p. 383]{T4} to the case of a maximal geodesic lamination with non-emtpy stump (see \S \ref{s:pairs} below for the definition of the stump of a geodesic lamination).
 
Let us return now to the definition of Thurston's asymmetric metric $K$.

It follows from the homogeneity property (Proposition \ref{homogeneity}) that  in the above definition of $K(g,h)$, the sup can be taken over the projectivized space $\mathcal{PML}$, instead of over the set of$\gamma$ in $\mathcal{S}$. In other words we also have
\[
K(g,h)=\log\sup_{[\mu]\in\mathcal{PML}}\frac{l_{h}(\mu)}{l_{g}(\mu)}
\]
where it is understood that in computing lengths, one takes the same representative $\mu$ of the projective class $[\mu]$ in the numerator and the denominator. 

In this new version of the definition of the function $K(g,h)$, taking the supremum over the space $\mathcal{PML}$ (instead of taking it over the set $\mathcal{S}$)
 has the advantage that, since space $\mathcal{PML}$ is compact, the supremum is realized. The existence of a projective measured lamination realizing this supremum is an important part in Thurston's development of his theory, and in fact, it is used in his proof of Theorem \ref{K=L}.

In the rest of this section, we review some properties of Thurston's asymmetric metric. We shall talk about its geodesics, and we start by recalling the definition of a geodesic in the case of an asymmetric metric, which is similar to that of a geodesic in a usual metric space.

A \emph{geodesic} in a space $(X,d)$ equipped with a distance function which might be asymmetric is defined as a map $h:I\to X$ where $I$ is an interval of $\mathbb{R}$, such that for every triple $x_1, x_2, x_3$ in $I$ satisfying $x_1\leq x_2\leq x_3$ we have $d(h(x_1),h(x_3))=d(h(x_1),h(x_2))+d(h(x_2),h(x_3))$. Notice that in the case where the distance function $d$ is asymmetric, one has to be careful about the order of the arguments when computing distances between two points. A map obtained from $h$ by reversing the time direction might not be a geodesic for $d$. In fact, most of the geodesics of Thurston's metric do not remain geodesic if the time direction is reversed.

\subsection{The topology and the metric properties of Thurston's metric}

There are several ways in which a non-symmetric distance function $d$ on a space $X$ can be symmetrized. Two common ways are:
\begin{itemize}
\item the \emph{arithmetic symmetrization} defined by  \[d_{\mathrm{arith}}(x,y)= \frac{1}{2}\left(d(x,y)+d(y,x)\right);\]

\item the \emph{max symmetrization} defined by 
\[d_{\mathrm{max}}(x,y)= \max\left(d(x,y),d(y,x)\right).\]
\end{itemize}

Both symmetrizations are genuine metrics, that is, they satisfy all the axioms of a metric in a usual sense. Furthermore, it is easy to see these two metrics are equivalent in the sense that there exist constants $C$ and $C'$ such that for all $x$ and $y$ in $X$, we have 
\[\frac{1}{C}d_{\mathrm{arith}}(x,y)\leq d_{\mathrm{max}}(x,y)\leq C' d_{\mathrm{arith}}(x,y)\] 
(One can take $C=1$ and $C'=2$). In particular, the two metrics induce the same topology on $X$.

By definition, the topology induced by an asymmetric metric is the one induced by (either of) these two symmetrizations. Herbert Busemann, who  studied extensively non-symmetric metrics, used this definition for the topology induced by such a metric provided it satisfies a condition of the form
\[d(x,x_n)\to 0\ \iff \ d(x_n,x)\to 0,\]
 see \cite[p. 2]{Busemann}. This condition is satisfied by Thurston's asymmetric metric, see Theorem \ref{th:4-2007a} below.

With this definition, it is easy to see that the topology induced by Thurston's asymmetric metric on Teichm\"uller space  coincides with the one induced by the  embedding of this space into the function space $\mathbb{R}_{+}^{\mathcal{S}}$ which we recalled in \S \ref{ss:T}, and therefore with the Teichm\"uller
metric. We shall refer to this topology as the ``usual topology" of Teichm\"uller space. 

The paper \cite{2007a} is dedicated to the metric properties of Thurston's metric. We review some results of that paper:

\begin{theorem}[Theorem 2 of \cite{2007a}]
For any sequence $g_n$ of elements in $\mathcal{T}(S)$, the following three properties are equivalent:

\begin{enumerate}
\item $g_n\to\infty$ in $\mathcal{T}(S)$, that is, this sequence leaves any compact set for $n$ large enough;
\item for all $h$ in $\mathcal{T}(S)$, we have $L(h,g_n)\to\infty$;
\item for all $h$ in $\mathcal{T}(S)$, we have $L(g_n,h)\to\infty$.
\end{enumerate}
\end{theorem}

\begin{theorem}[Theorem 4 of \cite{2007a}, see also \cite{Liu}]\label{th:4-2007a}
For any sequence $g_n$ in $\mathcal{T}(S)$ and for any element $g$ in this space, the following three properties are equivalent:

\begin{enumerate}
\item $g_n\to g$ in the usual topology;
\item  $L(g,g_n)\to 0$;
\item $L(g_n,g)\to 0$.
\end{enumerate}
\end{theorem}

In a space equipped with an asymmetric metric, there are two sorts of (open or closed) balls, namely, the left and the right balls. Let us recall the definitions. We restrict the notation to the case of $\mathcal{T}(S)$ equipped with the Thurston metric, since this is the case of interest in this survey.

Let $g$ be a point in $\mathcal{T}(S)$ and $R$ a positive number.
The \emph{left closed ball of center $g$ and radius $R$} is defined as
\[\{h\in\mathcal{T} \ \vert \ L(g,h)\leq R\}.
\]
The 
 \emph{right closed ball of center $g$ and radius $R$} is defined as
\[\{h\in\mathcal{T} \ \vert \ L(h,g)\leq R\}.
\]

   \emph{Left} and \emph{right open balls} are defined in the same way, by replacing the large inequalities by strict inequalities, and  \emph{left} and \emph{right spheres} are obtained by taking equalities instead of inequalities. 
   
   We note that for a general asymmetric metric, it is not always true that the closure of a left (respectively right) open ball is the union of such a ball with the left (respectively right) sphere centered at the same point and with the same radius. Busemann, in \cite[p. 3]{Busemann} gives conditions for this to hold.

\begin{proposition}[Proposition 5 of \cite{2007a}]\label{prop:5-2007a}
For any $g\in\mathcal{T}$ and $R>0$, the left and right closed ball of center $g$ and radius $R$ in $\mathcal{T}$ are compact for the usual topology.
\end{proposition}

The following is a reformulation of Proposition \ref{prop:5-2007a}:
\begin{proposition}[Proposition 8 of \cite{2007a}]\label{prop:compact}
Teichm\"uller space equipped with Thurston's asymmetric metric is proper, that is, left and right  open balls are compact with respect to the usual topology.
\end{proposition}

  The following corollary is a consequence of Theorem \ref{th:4-2007a}:

\begin{corollary}[Corollary 7 of \cite{2007a}]
The following three topologies on $\mathcal{T}(S)$ coincide:
\begin{enumerate}
\item The usual topology, that is, the topology induced by the metric $L_{\mathrm{max}}$;
\item the topology generated by the set of left open balls of the asymmetric metric $L$; 
\item the topology generated by the set of right open balls of the asymmetric metric $L$; 
\end{enumerate}
\end{corollary}

Spaces with asymmetric metrics were studied extensively by Busemann in his book \cite{Busemann}. His theory works for asymmetric metrics $d$ satisfying certain axioms which are all satisfied by Thurston's asymmetric metric (this uses in particular Theorem \ref{th:4-2007a} above). For the spaces he considered, Busemann introduced a notion of completeness which can be formulated as one form of completeness for genuine metrics (\cite[Chapter 1]{Busemann}):

\begin{definition}[Completeness] \label{def:complete}  A space $X$ equipped with an asymmetric metric $d$ is said to be \emph{complete} if for any sequence $(x_n)$ satisfying the following property:

 if for any $\epsilon >0$ there exists $n_0$ such that for any $n\geq n_0$ and any $m\geq n$ we have
  $d(x_n,x_{n+m})<\epsilon$, 

then the sequence $(x_n)$ converges to a point in $X$.
\end{definition}

Furthermore, Busemann showed that if in an asymmetric metric space any two points can be joined by a geodesic whose length realizes the distance between these points, then the following are equivalent:

\begin{enumerate}
\item The metric is complete (in the sense of Definition \ref{def:complete});
\item left closed balls are compact.
\end{enumerate}

This equivalence is a form of the Hopf-Rinow theorem for asymmetric metrics, see \cite[p. 4]{Busemann}.

From this equivalence and Proposition \ref{prop:compact}, we deduce the following:
\begin{corollary}[Corollary 10 of  \cite{2007a}]
Thurston's asymmetric metric is complete in the sense of Busemann.
\end{corollary}

\subsection{The Finsler structure} 
Since we are studying non-symmetric metrics, we need to deal with non-symmetric norms. Such a norm is used to describe the infinitesimal structure of a non-symmetric metric.
In this context, we call a \emph{weak norm} a structure on a finite-dimensional vector space which satisfies all the axioms of a usual norm except the symmetry axiom.
More precisely, we make the following
\begin{definition}
A \emph{weak norm} on a vector space $V$ is a function $V\to [0,\infty)$, $v\mapsto \Vert v\Vert$ such that for every $v$ and $w$ in $V$, we have:
\begin{enumerate}
\item  $\Vert v\Vert= 0\iff v=0$;
\item for every $t>0$, $\Vert tv\Vert =t\Vert v\Vert$;
\item for every $t\in [0,1]$, $\Vert tv+(1-t)w\Vert\leq t\Vert v\Vert+ (1-t)\Vert w\Vert.$
\end{enumerate}
\end{definition}
Given a weak norm $\Vert \ \Vert$ on a vector space $V$, its \emph{unit ball}  is the subset 
$B\subset V$ consisting of the vectors $v\in V$ satisfying $\Vert v\Vert \leq 1$.

The \emph{unit sphere} of a weak norm is the subset consisting of the vectors $v\in V$ satisfying $\Vert v\Vert = 1$. The definition of the unit sphere differs from the unit sphere of a usual norm in the fact that it is not necessarily symmetric with respect to the origin. This is a consequence of the fact that for $v\in V$, we do not necessarily have $\Vert v\Vert=\Vert -v\Vert$.

Like in the case of usual norms, the unit ball or the unit sphere of a weak norm completely determines this weak norm.

   A \emph{Finsler structure} on a differentiable manifold $M$ is a continuous family of weak norms parametrized by the tangent bundle of this manifold, such that each point $x$ of $M$, the weak norm is defined on the tangent space of $M$ at $x$. The length of a $C^1$ path in $M$ in such a manifold is defined as the integral over this path of the norms of the tangent vectors to this path, measured at each point using the weak norm provided by the Finsler structures. A distance function is then obtained on $M$, by setting the distance between two points to be the infimum of the lengths of piecewise $C^1$ paths joining them. Such a weak definition of a Finsler structure is studied in the paper \cite{PT}.

Thurston showed that the asymmetric metric $K$ (or $L$) that we recalled in \S \ref{ss:Thurston} on the Teichm\"uller space $\mathcal{T}(S)$ is Finsler, that is, it is a length metric associated with a Finsler structure.
He also gave an explicit formula for the weak norm on each tangent space, namely, for any a tangent vector $v$ at a point $g$ in $\mathcal{T}(S)$, we have

\[
\Vert v \Vert =\sup_{\lambda \in \mathcal{ML}} \frac{d \ell_{\lambda}(v)}{\ell_{\lambda}(g)}.
\]
Here, $\mathcal{ML}$ is as before the space of measured laminations on the surface, $\ell_{\lambda}:\mathcal{T}(S)\to \mathbb{R}$ is the
 geodesic length function on Teichm\"uller space associated with the measured lamination $\lambda$ and $d \ell_{\lambda}$ is the differential of  $\ell_{\lambda}$ at the point $g\in \mathcal{T}(S)$.

An analogous formula for the Finsler norm of the Teichm\"uller metric  is proved  in the paper
\cite{2012e}. It is shown there that for any vector $v$ in the tangent space to Teichm\"uller space at a point $g$, the norm of $v$ for the Teichm\"uller metric (which is known to be Finsler) is given by
\[\Vert v\Vert = \sup_{\lambda \in \mathcal{ML}}\frac{d\mathrm{Ext}_{\lambda}^{1/2} (v)}{\mathrm{Ext}_{\lambda}^{1/2} (g)}
.\]
In this formula, for a given element $\lambda$ in $\mathcal{ML}$, $\mathrm{Ext}_\lambda: \mathcal{T}(S)\to \mathbb{R}$ is the associated extremal length function on Teichm\"uller space.  The formula is proved in the paper \cite{2012e}  (Theorem 3.6). In the same paper,  several other analogies between the Teichm\"uller metric and Thurston's metric are pointed out (some of them classically known and others are new and proved in that paper).

Thurston also gave a description of the unit sphere associated with the Finsler structure of the metric $K$ in the cotangent space at each point of Teichm\"uller space. This unit sphere is the boundary of a convex body which is the dual of the unit ball of the weak norm associated with the Finsler structure  of the Thurston metric. It turns out that this unit sphere in the cotangent space at each point $g$ of $\mathcal{T}(S)$ is the image of  projective measured lamination space $\mathcal{PML}$ by the map
\[ [\lambda]\mapsto d \log l_g(\lambda) \]
 which assigns to each element $[\lambda]$ of $\mathcal{PML}$ the logarithmic derivative $d \log l_g(\lambda)$ of the length of a representative $\lambda$ of $[\lambda]$. This logarithmic derivative, considered as an element of the cotangent space to Teichm\"uller space at $g$, is invariant by the action of the positive reals on $\mathcal{ML}$  and therefore depends only on the projective class $[\lambda]$ of $\lambda$.

An analogous description for the unit sphere associated with the Finsler structure of the Teichm\"uller metric, in the cotangent space at each point of Teichm\"uller space, is given in the paper \cite{2012e} (Theorem 4.1).  

We note that the fact that the the embedding of projective measured lamination space $\mathcal{PML}$ in the cotangent space to Teichm\"uller space by the map
$ [\lambda]\mapsto d \log l_g(\lambda)$ is convex and contains the origin of the vector space is part of the properties of Finsler weak norms and co-norms, but that in the case of the Thurston metric, 
Thurston proved directly the convexity of this image before introducing the Finsler structure, see \cite[Theorem 5.1]{Thurston1986}. Furthermore, Thurston showed that this unit ball is not strictly convex, and he described in some detail its facets, in particular those of minimal dimension, in terms of the geometric properties of the laminations that they represent. The  lack of strict convexity of the unit ball of the weak co-norm has an interpretation as a lack of strict convexity of the unit ball of the weak norm associated with the Finsler norm at each tangent space. From the point of view of the metric properties of the Thurston metric on Teichm\"uller space, it expresses the possibility of having multiple geodesics between some pairs of two points in that space (see \S \ref{ss:Thurston}). One may think here of the metric on $\mathbb{R}^n$ associated with the $L^1$ norm on $\mathbb{R}^2$, called the taxicab metric (the terminology is possibly due to Thurston), where there is a abundance of geodesics between any two distinct points.

The Finsler structures of Teichm\"uller space equipped with its  Teichm\"uller metrics and Thurston are rare instances of Finsler structures where the unit balls have such an appealing geometric significance.
 
\subsection{Thurston's cataclysm coordinates and stretch lines}\label{ss:cataclysm}

Let $\mu$ be a maximal geodesic lamination and let $\mathcal{MF}(\mu)$ be the subset of $\mathcal{MF}$ consisting of the equivalence classes of compactly supported measured foliations that admit representatives that are transverse to $\mu$.   (We recall that all the measured foliations that we consider are compactly supported, see \S \ref{ss:foliations}.) The set $\mathcal{MF}(\mu)$ does not depend on the choice of the hyperbolic metric for which the lamination $\mu$ is geodesic. The reason is that for any element of $\mathcal{MF}$, the fact of admitting a representative transverse to $\mu$ is a topological property, and the realizations of $\mu$ as a geodesic lamination for various metrics are all isotopic.  

The subset $\mathcal{MF}(\mu)$ of $\mathcal{MF}$ is invariant by the natural action of $\mathbb{R}_+$ on this space. We denote by  $\mathcal{PMF}(\mu)$ its projectivization. 

By fixing the geodesic lamination $\mu$ and varying the hyperbolic structure, we get a map 
\[
\phi_{\mu}:\mathcal{T}\to \mathcal{MF}(\mu).
\]
which associates to each equivalence class of hyperbolic structures $g$ on $S$ the corresponding  equivalence class of the horocyclic measured foliations.
The following is a result of Thurston, see \cite[Propositions 9.2 and 10.9]{Thurston1986}:
\begin{theorem}\label{th:cataclysm}
The map $\phi_{\mu}$ is a homeomorphism. 
\end{theorem}
 
 The map $\phi_{\mu}$ is a system of coordinates on Teichm\"uller space which Thurston calls  \emph{cataclysm coordinates} \cite[\S 9]{Thurston1986}. This system of coordinates depends on the choice of a maximal geodesic lamination $\mu$. Theorem \ref{th:conv} in \S \ref{s:stretch} below tells us that in a certain precise sense this system of coordinates extends to Thurston's boundary of Teichm\"uller space; in fact it leads to coordinate systems near points in the boundary.

  Now we can define stretch lines.

 As a subset of Teichm\"uller space, a \emph{stretch line} is a family of hyperbolic structures which is
 the inverse image by $\phi_{\mu}$ of an orbit of the natural action of the positive real numbers on this space $\mathcal{MF}(\mu)$.    We say that the lamination $\mu$ is the \emph{support} of this stretch line.
 Furthermore, a stretch line is equipped with a special parametrization. The precise definition is as follows: Given a hyperbolic structure $g$ on $S$, we call a \emph{stretch line with support $\mu$ and origin $g$}  a mapping $h:\mathbb{R}\to \mathcal{T}$ defined, for any $t$ in $\mathbb{R}$, by
\[h(t)=\varphi_{\mu}^{-1}(e^tF_\mu(g)).\]

Equipped with this parametrization, the stretch line $h$ is a geodesic for Thurston's metric $K$ (or $L$), see \cite[\S 8]{Thurston1986}.

By restricting a stretch line to the set of nonnegative numbers, we obtain the notion of \emph{stretch ray starting at a hyperbolic structure $g$}. 

The hyperbolic structure $g$ and the maximal geodesic lamination $\mu$ determine the stretch line $h:\mathbb{R}\to\mathcal{T}$ together with its origin $h(0)=g$. The pair $(\mu, g)$ determines the horocyclic foliation $F_{\mu}(g)$.  
  Conversely, any pair $(\mu, F)$ where $\mu$ is a geodesic lamination for some hyperbolic metric on $S$ and  $F$  a compactly supported measured foliation transverse to $\mu$ determines a complete hyperbolic structure $g$ for which $F$ is the horocyclic foliation and a stretch line with origin $g$.
 
The fact that a stretch line is a geodesic for Thurston's metric and that it is determined by a pair $(\mu,F)$ where $\mu$ is a lamination and $F$ a measured foliation transverse to it is 
reminiscent of the fact that a geodesic for the Teichm\"uller metric is determined by two transverse measured foliations, its horizontal and vertical foliations. Note however that the geodesic lamination $\mu$ associated with a stretch line is not necessarily measured (it may not carry a transverse measure of full support), whereas both foliations associated with a Teichm\"uller geodesic are measured.

In the special case where the surface $S$ has cusps and where the geodesic lamination $\mu$ is an \emph{ideal triangulation}, that is, when $\mu$ consists of a finite collection of bi-infinite geodesics which converge in both directions to cusps of the surface, then, we can parametrize both the Teichm\"uller space of $S$ and the space $\mathcal{MF}(\mu)$ of measured foliations transverse to $\mu$ by the shift coordinates on the edges of $\mu$ (see \S \ref{s:ideal}). The cataclysm coordinates are then those given by the identity map in terms of these shift coordinates. A stretch line, in these coordinates, amounts to multiplying by a constant factor the shift coordinates (at time $t$, the shift coordinates are multiplied by the factor $e^t$).

 \subsection{Stretch lines and geodesics for Thurston's metric}

 Thurston proved that any two points in Teichm\"uller space can be joined by a geodesic which is a finite concatenation of segments of stretch lines. He gave constructive ways to find such geodesics, based on the study pf Lipschitz maps between ideal triangles.
He also showed that in general, such a geodesic joining the two points is not unique. This contrasts
with the case of the Teichm\"uller metric on Teichm\"uller space for which the geodesic joining any two points is unique.

One should note that there is a great variety of geodesics
for Thurston's metric that have geometric explicit descriptions and that are not concatenations of segments of stretch lines. Some of them are made explicit in the papers \cite{2009j}  and \cite{2015-Yamada1} and are based on the construction of Lipschitz maps between right-angled hyperbolic hexagons.

\section{On the asymptotic behavior of  stretch lines}\label{s:stretch}

In this section, we study the convergence of stretch lines to points on Thurston's boundary of Teichm\"uller space. We have the following:

\begin{theorem} \label{th:conv-stretch} Let $\mu$ be a maximal geodesic lamination on the surface $S$ equipped with a hyperbolic metric and let $F_{\mu}(h)$ be the associated horocyclic foliation. We assume that  $F_{\mu}(h)$ does not consist entirely of leaves that are closed leaves homotopic to punctures (in other words, we assume that the geodesic lamination associated with $F_{\mu}(h)$ is not empty). Then the stretch line $h_t:\mathbb{R}\to\mathcal{T}$ with support $\mu$ and starting at $h_0=h$ converges, as $t\to\infty$, to the projective class $[F_{\mu}(h)]$ as a point on the boundary of Teichm\"uller space.
\end{theorem}

Motivated by this result, the horocyclic measured foliation associated with a stretch line (or the projective class of this  horocyclic measured foliation) is called the \emph{direction} of the stretch line.

  Theorem \ref{th:conv-stretch} is a consequence of the following theorem:

\begin{theorem}[\cite{1991a}, Theorem 4.1] \label{th:conv} Let $\mu$ be a maximal geodesic lamination on $S$. 

(1) Let $g_n$ be a sequence of elements in $\mathcal{T}(S)$ which converges to a point $[F]$ on Thurston's boundary of Teichm\"uller space and suppose that $[F]$ is in $\mathcal{PMF}(\mu)$. Then the associated sequence of projective horocyclic foliations $[F_\mu(g_n)]$ also converges to $[F]$.

(2) Let $g_n$ be a sequence of elements in $\mathcal{T}(S)$ that tends to infinity, and assume that the associated sequence $F_\mu(g_n)$ tends to infinity in $\mathcal{MF}$. If the associated sequence of projective measured foliations $[F_\mu(g_n)]$ converges in $\mathcal{PMF}$, then the sequence $g_n$ also converges and the two sequences have the same limit.
\end{theorem}

This theorem is used in the paper \cite{1991a} to show that the earthquake flow in Teichm\"uller space, if its associated measured geodesic lamination is maximal, extends to a flow on Thurston's boundary of Teichm\"uller space and it induces there a flow which may be called an \emph{earthquake flow on projective measured geodesic lamination space}.

  The proof of Theorem \ref{th:conv} involves several steps, the most important one being Lemma \ref{lemma:fund} below, referred to in \cite{1991a} as the ``Fundamental Lemma." This lemma gives precise estimates that make relations between length functions associated with the metrics $g_n$ and intersection functions with the horocyclic foliations $F_\mu(g_n)$.
  Before stating this lemma, we introduce some notation.

  Given $\epsilon >0$ and a maximal geodesic lamination $\mu$, we define $V(\mu,\epsilon)$ to be the subset of Teichm\"uller space consisting of the hyperbolic structures stisfying $l_g(\mu)>\epsilon$.

\begin{lemma}[The Fundamental Lemma, \cite{1991a}, Lemma 4.9]\label{lemma:fund}
Consider a sequence $(g_n)_{n\geq 0}$ of elements in $V(\mu,\epsilon)$ converging to a point in $\mathcal{PMF}(\mu)$ and let $F_n=F_\mu(g_n)$ be the associated sequence of projective horocyclic foliations. 
Then, for every $\alpha$ in $\mathcal{S}$, there exists a constant $C(\alpha)$ (depending only on $\alpha$) such that for any $n\geq 0$, we have 
\[i(F_n,\alpha)\leq l_{g_n},(\alpha)\leq i(F_n,\alpha)+ C(\alpha).\]

\end{lemma}

 We shall prove a weaker version of Lemma \ref{lemma:fund} that is valid in the case where $\mu$ is a finite geodesic lamination, and we first define this notion:  
 
 A \emph{finite} geodesic lamination on $S$ is a geodesic lamination in which each leaf is of one of the following two types:

\begin{itemize}
\item a simple closed geodesic;

\item
a bi-infinite geodesic whose ends either spiral along a simple closed geodesic or converge to a cusp of $S$.
\end{itemize}

A finite geodesic lamination decomposes the surface $S$ into a finite number of ideal hyperbolic triangles such that $S$ is obtained by gluing these triangles along their boundary components.

 In the case where $\mu$ is a finite geodesic lamination, we have the following more precise version of the Fundamental Lemma.
  
\begin{proposition}[\cite{1988}, Proposition 3.1]\label{lemma:fund2}
Let $\mu$ be a finite geodesic lamination. Then, for every $\alpha$ in $\mathcal{S}$ and for every $g$ in  $\mathcal{T}(S)$, we have 
\[  i(F_\mu(g),\alpha)\leq l_g(\alpha)\leq i(F_\mu(g),\alpha)+4 i(\alpha,\mu)
.\]

\end{proposition}
 
We now prove Proposition \ref{lemma:fund2}. We start with the following lemma:

\begin{lemma}\label{lem:proj}
Let $l$ be a geodesic segment contained in the support of the horocyclic foliation of an ideal triangle. Then the projection of $l$ along the leaves of this foliation on any of the two boundary components of the ideal triangle is length-decreasing.
\end{lemma}
Such a geodesic segment $l$ is a subsegment of one of the three segments represented in Figure \ref{fig:3cases}.
Note that for the proof of Lemma \ref{lem:proj}, it suffices to assume that the geodesic segment is transverse to the horocyclic foliation.
With this in mind, the proof of this lemma is then a computation that can easily be done in the upper half-plane model of the hyperbolic plane, using the formula we gave in \S \ref{ss:hyp} for the infinitesimal length element in that model.

\begin{proof}[Proof of the first inequality of Proposition \ref{lemma:fund2}]
Consider the closed geodesic in the class $\alpha$. There are three possibilities for the connected components of the intersection of this closed geodesic with the support of the horocyclic foliation of an ideal  triangle; they are represented in Figure \ref{fig:3cases}. These three cases correspond to the fact that such a segment has 2, 1 or 0 intersection points respectively with the boundary of the hyperbolic triangle. We modify each such connected component in the following way: In Cases (1) or (2),  we replace this segment, keeping its endpoints fixed, by a segment which is the concatenation of a segment contained in a leaf of the horocyclic foliation and a segment contained in the boundary of the ideal triangle. (Such a segment, for obvious reasons, is called a \emph{horogeodesic segment.}) In case (3), we push the segment in the non-foliated region. By Lemma \ref{lem:proj}, the total transverse measure of the segments obtained by this modification with respect to the horocyclic foliation $F_\mu(g)$ is bounded above by the length $l_\alpha$. Adding the transverses of all the new segments, we obtain $i(F_\mu(g),\alpha)\leq I(F_\mu(g),\alpha)\leq l_g(\alpha)$, which is the inequality we wanted.\end{proof}

\begin{figure}[htbp]
 \psfrag{(1)}{\small $(1)$}
  \psfrag{(2)}{\small $(2)$}
   \psfrag{(3)}{\small $(3)$}
\centering

\includegraphics[width=12cm]{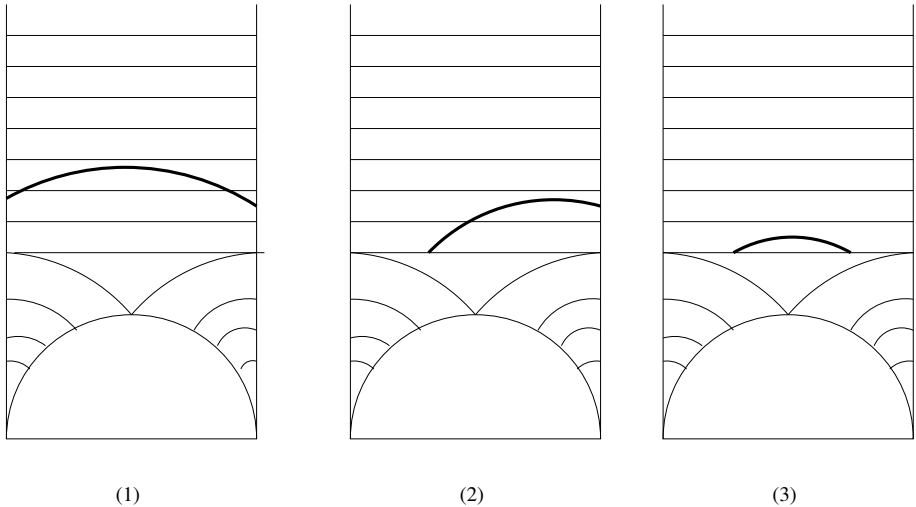}

\caption{\small The three cases for the connected components of the intersection of a geodesic segment with the support of the horocyclic foliation on an ideal triangle}
\label{fig:3cases}
\end{figure}

\begin{proof}[Proof of the second inequality of Proposition \ref{lemma:fund2}]

We start by representing $\alpha$ by a closed curve $\alpha'$ which is quasi-transverse to the horocyclic measured foliation $F_\mu(g)$ (see \S \ref{ss:quasi-transverse}). For the proof of the proposition, we may assume, without loss of generality, that $\alpha'$ is transverse to $F_\mu(g)$. Indeed, if $\alpha'$ has sub-segments contained in the leaves of $F_\mu(g)$, then the foliation $F_\mu(g)$ has a connection between singular points, but the set of measured foliations of the form $F_\mu(g)$ (with $g$ varying) that have no connections between singular points is dense in the space of all measured foliations that are of the form $F_\mu(g)$ for some $g$. Then, by continuity, it suffices to prove the desired inequality in the generic case where there are no such connections. 

Next, we can assume that $\alpha'$ has minimal number of intersection points in its homotopy class with $\mu$, that is, we assume that the number of intersection points between $\alpha'$ and $\mu$ is equal to $i(\alpha,\mu)$.
This can be done by eliminating all the discs embedded in $S$ whose boundaries consist of the union of an arc contained in a leaf of $\mu$ and an arc contained in $\alpha'$. The elimination is done keeping the transversality between the simple closed curve and the horocyclic measured foliation. This uses the fact that if there exists such a disc, then there exists an innermost one which is of one of the two tupes represented in Figure \ref{fig:minimal}, therefore we can eliminate all these discs by performing an isotopy on $\alpha'$, working at each step with innermost discs and making an induction on the number of intersection points between $\alpha'$ and $\mu$.
\begin{figure}[htbp]

\centering

\includegraphics[width=12cm]{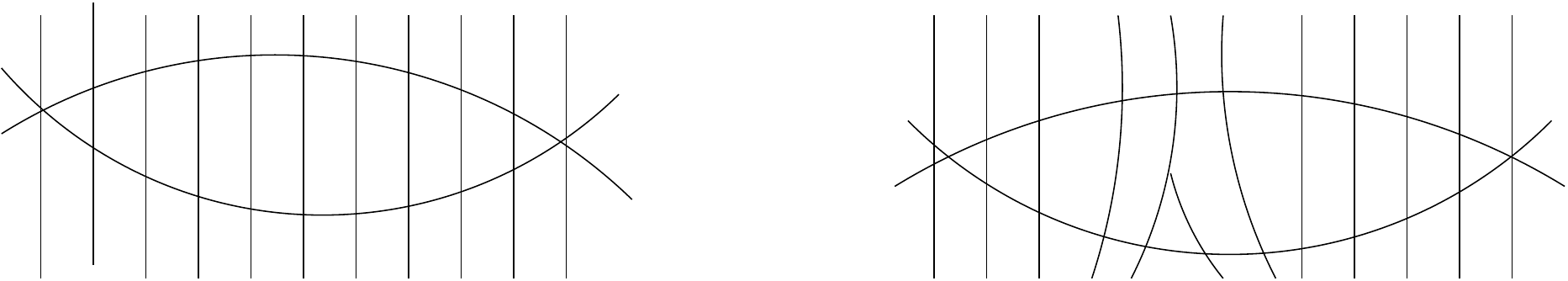}

\caption{\small A minimal disc whose boundary consists of an arc in a leaf of $\mu$ and an arc in $\alpha'$ that are both transverse to $F_\mu(g)$}
\label{fig:minimal}
\end{figure}

After these reductions, we let $a_1,\ldots,a_m$ be the connected components of the intersection of $\alpha'$ with the support of the horocyclic foliation of ideal  triangles in the complement of $\mu$, and $b_1, \ldots, b_k$ the connected components of the intersection of the same curve with the complement of the support of $F_\mu(g)$. We call these two classes of segments ``of type $a_i$" and ``of type $b_i$" respectively.

 We now replace each segment of type $b_i$ by a geodesic segment, keeping unchanged its extremities. The length of a resulting arc is bounded by 1 (which is the diameter of a non-foliated region).

The arcs of type $a_i$ are of type (1) or (2) represented in Figure \ref{fig:3cases}, with the additional property that they are transverse to the foliation $F_\mu(g)$. The number of segments of type $a_i$ is bounded above by $2i(\alpha,\mu)$. Since each segment $b_i$ is adjacent from each side to a segment of type $a_i$, the total number of segments of type $b_i$ is bounded above by the total number of segments of type $a_i$, therefore it is bounded above by $2 i(\alpha,\mu)$. Thus, the total length of all the arcs of type $b_i$ is bounded above by $2i(\alpha,\mu)$.

We replace now each arc $a_i$ keeping its endpoints fixed by a horogeodesic arc consisting of the concatenation of an arc $c_i$ contained in a leaf of the horocyclic foliation and an arc $d_i$ contained in a leaf of $\mu$. We let $\alpha''$ be the resulting horogeodesic curve in the homotopy class $\alpha$.

The number of arcs $c_i$ is bounded above by $2i(\alpha,\mu)$, therefore the total lengths of these segments is bounded above by $2i(\alpha,\mu)$.

Putting together these estimates, we get
\[l_g(\alpha'')\leq i(F_\mu(g),\alpha'')+4i(\alpha,\mu).\]
Since $l_g(\alpha)$ is bounded above by $l_g(\alpha'')$, we get the second inequality of Proposition \ref{lemma:fund2}.
This completes the proof of Proposition \ref{lemma:fund2}.
\end{proof}

The proof of the Fundamental Lemma (Lemma \ref{lemma:fund}) uses the same ideas as those of Proposition \ref{lemma:fund2}, but it holds for a general maximal geodesic lamination  $\mu$ instead of the finite geodesic lamination $\mu$ in the lemma. The arguments involving the connected components of the intersection of a simple closed curve with the complementary components of the lamination are replaced by the connected components of the intersection of a simple closed curve with a set of rectangles that are adapted to the lamination.

We conclude this section by giving more information on the convergence of a family of hyperbolic structures along a stretch line. 

Recall that a hyperbolic structure is determined by its marked simple length spectrum, that is, the collection of lengths of simple closed geodesics representing the elements of the set $\mathcal{S}$. Thus, information about the behavior of lengths of simple closed curves (and more generally of measured laminations) under a stretch line is also information about the hyperbolic structures that vary under a stretch line.
   The following result by Th\'eret gives information on the behavior of lengths of general measured geodesic laminations  under stretch lines. 

\begin{theorem}[Th\'eret, \cite{Theret-Thesis}, Theorem 3  p. 25 and Theorem 2 p. 62] \label{th1}   
Let $h:\mathbb{R}\to \mathcal{T}(S)$, $t\mapsto h(t)=h_t$ be a stretch line passing by a hyperbolic metric $h=h(0)$ and supported by a geodesic lamination $\mu$.
Let $L_{\mu}(h)$ be the measured geodesic lamination that corresponds to the horocyclic measured foliation $F_{\mu}(h)$. Then, for any measured geodesic lamination $\lambda$ on $S$, we have:
\begin{enumerate}
\item $\lim_{t\to\infty}l_{h_{t}}(\lambda)=0$ if and only if the support of $\lambda$ is contained in the support of $L_{\mu}(h)$.
\item $\lim_{t\to\infty}l_{h_{t}}(\lambda)=+\infty$ if and only if $i(\lambda,L_{\mu}(h))\not= 0$.

\item  If  the support of  $\lambda$ is disjoint from  the support of  $L_{\mu}(h)$, then, for $t\geq 0$, $l_{h_{t}}(\lambda)$ is bounded from above and from below by strictly positive constants.
\end{enumerate}

\end{theorem}
 The proof of (1) uses the following result, which is proved in \cite[Theorem 1  p. 45]{Theret-Thesis}: Under the hypotheses of the theorem, suppose that the horocyclic measured lamination $L_{\mu}(h)$ is not empty, that is, that the horocyclic measured foliation does not consist entirely of closed leaves parallel to the cusps. Then there exists a positive constant $A$ such that for any $t\geq 0$, we have
 \[l_{h_{t}}(L_{\mu}(h_t))\leq Ae^{t-e^{t}}
.\]
This implies that 
 \[ \lim_{t\to\infty}l_{h_{t}}(L_{\mu}(h))=0
 \]
 and that this convergence is exponential. The proof of this fact uses the notion of the length of a measured foliation (see Definition 3.2 in \cite{1991a}), and the fact for each $t\geq 0$ one can find a partial measured foliation $F_t$ in the same equivalence class as $F_{\mu}(h_t)$ whose length converges to 0 as $t\to\infty$.

 Note that in the case where the surface $S$ has cusps and where $\mu$ is an ideal triangulation, if the hyperbolic structure $h$ corresponds to a \emph{symmetric gluing} of the ideal triangles that compose the complement of $\mu$ (that is, if all the shift parameters are zero), then the horocyclic foliation consists entirely of closed leaves parallel to the cusps, and there is no stretch line starting at $h$ supported by $\mu$ (in fact, stretching along such a $\mu$ does not change the hyperbolic structure).
 
Conclusion (3) in Theorem \ref{th1} says that if the support of $\lambda$ is disjoint from the support of the measured geodesic lamination $L_{\mu}(h)$ then, along the positive part of the stretch line,  its length can vary only by  a bounded amount.

There is an analogue of  Theorem \ref{th1} which concerns the behavior of length functions of measured laminations along anti-stretch lines (stretch lines traversed in the negative direction), and which is also due to Th\'eret. We shall review this in Theorem \ref{th2} below.

The proof of Theorem \ref{th1}, which covers a substantial part of the paper \cite{Theret-Thesis}, involves in particular a generalization of the double inequality in the Fundamental Lemma (Lemma \ref{lemma:fund}) to the case where the simple closed curve $\gamma$ is replaced by a general measured geodesic lamination; cf. \cite[Corollary 2, p. 61]{Theret-Thesis}.

\section{The relative asymptotic behavior of pairs of stretch lines}\label{s:pairs}
In this section, we review some results due to Th\'eret on the asymptotic behavior of pairs of stretch lines. They concern  questions of parallelism and divergence of such pairs.
Before stating the results, we need to recall some terminology.
We start with the following notion introduced by Thurston in his paper \cite{Thurston1986}:

The \emph{stump} of a geodesic lamination is the maximal (in the sense of inclusion) sublamination that can be equipped with some transverse measure. (Recall that, by definition, a transverse measure  for a lamination must be of full support and finite on compact sets).

Any geodesic measured lamination $\mu$ which is non complete can be completed (in a non-unique way) into a maximal geodesic lamination whose stump is $\mu$, by adding infinite geodesics that, in each direction, either converge to a cusp of the surface or spiral along leaves of $\mu$. Here, the expression ``spiraling along leaves of $\mu$," for a given half-leaf of an infinite geodesic means that either this half-leaf spirals along a simple closed geodesic of $\mu$ (the operation we considered in \S \ref{gluing}), or that the half-leaf enters a cusp of some component of $S\setminus \mu$. Another way of expressing the fact that a half-leaf $l$ of a bi-infinite geodesic on a surface spirals along a half-leaf $l'$ of a geodesic lamination $\mu$  is to say that the two geodesic rays $l$ and $l'$ have lifts to the universal cover of the surface that converge to the same point at infinity.

The stump of a geodesic lamination $\mu$ is empty if and only if the lamination consists of non-compact leaves that converge in both directions towards cusps of the surface. A maximal geodesic lamination with empty stump is an ideal triangulation.
 
The \emph{stump} of a stretch line is the stump of the maximal geodesic lamination that directs it.

A \emph{multicurve} on $S$ is a collection of disjoint essential simple closed curves no two of which are homotopic.  A \emph{weighted muticurve} is a multicurve equipped with a positive weight for each component. A weighted multicurve gives a well-defined element of $\mathcal{LM}$ (or, equivalently, of $\mathcal{MF}$).  

We also recall that with a stretch line on a hyperbolic surface is associated  a horocyclic foliation called its direction (\S \ref{s:stretch}). A stretch line is then said to be \emph{cylindrical} if its direction has only compact leaves. Thus, the direction of a cylindrical stretch line, as a measured foliation, is associated with a weighted multicurve of $S$.

A geodesic lamination is said to be \emph{chain recurrent}  if it is the limit, in the Hausdorff topology of compact subsets of the surface, of a sequence of geodesic laminations that are multicurves.

Finally, we also need the following notions introduced by Th\'eret \cite{T1}:

Two stretch lines $h:[0,\infty)\to \mathcal{T}$ and $g:[0,\infty)\to \mathcal{T}$ are said to \emph{diverge} if $d(g_t,h_t)\to\infty$ and $d(h_t,g_t)\to\infty$ as $t\to\infty$.

Two stretch lines $h:[0,\infty)\to \mathcal{T}$ and $g:[0,\infty)\to \mathcal{T}$ are said to be \emph{parallel} if, up to a parameter change, the two distances $d(g_t,h_t)$ and $d(h_t,g_t)$ are bounded  as $t\to\infty$.

  Th\'eret proved the following results on cylindrical stretch lines.

\begin{theorem}[\cite{T1}, Theorem 5.2]  \label{th:5.2} Two cylindrical stretch lines in Teichm\"uller space whose directions, as projective classes of measured foliations, are distinct, diverge. 
\end{theorem}

\begin{theorem}[\cite{T1}, Theorem 6.1]  \label{th:5.3} Two cylindrical stretch lines  in Teichm\"uller space, with some appropriate parametrization, are parallel if and only if they converge to the same point on Thurston's boundary.
\end{theorem}

The proofs of Theorems \ref{th:5.2} and \ref{th:5.3} are based on quantitative estimates of the behavior at infinity of the lengths of the cylinders that are the components of the direction of the cylindrical stretch lines. These estimates are precise versions of the convergence in Theorem \ref{th1}, and they  include the following result:
\begin{theorem}[Th\'eret, Theorem 5.1 of \cite{T1}]
Let $t\mapsto h_t$ be a cylindrical stretch line whose support $\mu$ is chain-recurrent geodesic lamination and let $\lambda$ be a component of the direction of $h_t$, which is a simple closed curve $\gamma$ weighted by a real number $w$. Then, we have, as $t\to\infty$,
\[l_{h_{t}}(\lambda)\sim 2\sqrt{k}e^{-e^{t}w/2}
,\]
where $k$ is the product of two numbers, each of them being the number of non-foliated regions of the horocyclic foliation which are adjacent to one of the two boundaries of the maximal cylinder with core curve $\gamma$ contained in the direction of the stretch line.

\end{theorem}

Another result about cylindrical stretch lines is the following:

\begin{theorem}[Th\'eret, Theorem 6.2 \cite{T1}] \label{t:LM} Two cylindrical stretch lines are parallel if and only if they have the same direction.
\end{theorem} 
The proof of the only if part of Theorem \ref{t:LM} uses the estimates in Theorem \ref{th1}, which show that if the two cylindrical stretch lines $g_t$ and $h_t$ are directed by two topologically distinct measured laminations $\lambda$ and $\lambda'$,  the two ratios $l_{h_{t}}(\lambda)/ l_{g_{t}}(\lambda)$ and $l_{g_{t}}(\lambda)/ l_{h_{t}}(\lambda)$ 
tend to infinity, therefore the two stretch lines diverge.

The proof of the if part is based on Theorem \ref{th:5.2}.

Theorem \ref{t:LM} should be contrasted with a result of Masur proved in \cite{Masur}, which says that there exist parallel cylindrical Teichm\"uller rays whose directions, as measured foliations, are distinct as projective classes.

Finally, we state the following result on cylindrical stretch lines:

\begin{theorem}[Th\'eret \cite{T1}, Theorem 5.3] \label{th:cylindrical-anti-stretch} Let  $h:[0,\infty]\to \mathcal{T}$ be a cylindrical stretch line. Then for any $c>0$, we have 
\[ \lim_{t\to\infty} d(h_{t+c}, h_t)=\infty.\]

\end{theorem}

The statement of Theorem \ref{th:cylindrical-anti-stretch} is another illustration of the fact that the metric $K$ is not symmetric; in fact, its restriction to the given stretch line is not symmetric, since we have, for all $t\geq 0,$ $d(h_t,h_{t+c})=c$.

\section{On the asymptotic behavior of anti-stretch lines}\label{s:anti}

In this section, we study convergence of \emph{anti-stretch lines}, that is, stretch lines traversed in the reverse direction. 
In general, these lines are not geodesics for Thurston's metric but ther are geodesics for the asymmetric metric we call the \emph{reverse Thurston metric} $K'$ on Teichm\"uller space, defined by 
\[K'(x,y)=K(x,y)\] for $x$ and $y$ in $\mathcal{T}(S)$, where $K$ is the Thurston metric.
  We shall not be precise about the exact parametrization of an anti-stretch line, because we shall only be interested in its limiting behavior, that is, when the parameter goes to infinity.

Anti-stretch lines are not necessarily geodesics for Thurston's metric, even after reparametrization. In fact, there is no hope of obtaining a general result saying that an anti-stretch line converges to the geodesic lamination that supports it, since this geodesic lamination is not necessarily a measured lamination, that is, generally speaking, it does define an element of Thurston's boundary of Teichm\"uller space.
There are cases however where we can prove that an anti-stretch line converges to a point on Thurston's boundary. The first result in this direction was obtained in \cite[Proposition 5.2]{1991a} and in  a more general form in Theorem \ref{pro:uniquely} below.

We recall that a measured foliation, or a measured lamination, is said to be \emph{uniquely ergodic} of it carries only one transverse measure up to a multiple constant.
 
The results in this section are due to Th\'eret.

\begin{theorem}[Th\'eret, Theorem 3.2 of \cite{T4}]\label{pro:uniquely} Assume that the support of a stretch line is a geodesic lamination whose stump $\gamma$ is uniquely ergodic. Then, the corresponding anti-stretch line converges to the point on Thurston's boundary which is the projective class $[\gamma]$.
\end{theorem}

The proof of Theorem \ref{pro:uniquely} that is given in \cite{T4} is based on a description of the set of cluster points in the compactified Teichm\"uller space of the set of points of the negative part of a stretch line with non-empty stump $\gamma$. The set of cluster points turns out to be the set of  projective classes of measured geodesic laminations that are topologically contained in $\gamma$ (Corollary 3.1 of \cite{T4}).

A special case  of Theorem \ref{pro:uniquely}, namely, the case where  the support $\mu$ of the stretch  carries a transverse measure of full support which is uniquely ergodic, was obtained in \cite[Proposition 5.2]{1991a}.

An immediate consequence of Theorems \ref{th:conv-stretch} and  \ref{pro:uniquely} is the following:

\begin{corollary} 
Let $P$ and $Q$ be two points in Thurston's boundary $\mathcal{PMF}$ of Teichm\"uller space that can be represented by two transverse measured foliations on the surface, and assume that $P$ is uniquely ergodic. Then there is a geodesic in Teichm\"ulller space, namely, a stretch line, which converges in the negative direction to $P$ and positive direction to $Q$.
\end{corollary}
 
Th\'eret calls a stretch line \emph{elementary} if its stump is a union of simple closed geodesics. In the paper \cite{T3}, he studies the convergence to  points in Thurston's boundary of elementary anti-stretch lines (that is elementary stretch lines traversed in the reverse diresction). The result is stated in Proposition \ref{th:Theret-elementary} below. Before this, we make the following definition: 

Consider a maximal elementary geodesic lamination $\mu$. We endow each of the closed geodesics in its stump with the counting transverse measure, that is, with the transverse measure given by counting the number of intersection points. The projective
class of this transverse measure is called the \emph{barycenter} of the stump of $\mu$.

\begin{theorem}[Th\'eret, \cite{T3} Theorem 4.1] \label{th:Theret-elementary}
Any elementary anti-stretch line negatively converges to a point on Thurston's boundary which is the barycenter of the stump of the maximal geodesic lamination that supports it.
 \end{theorem}

An elementary stretch line whose stump consists of simple closed geodesics that are in more than one homotopy class is an example of a a stretch line with a non-uniquely ergodic stump. Thus, Theorem \ref{th:Theret-elementary} gives examples of negatively converging stretch lines whose stump is not uniquely ergodic. 

The proof of Theorem \ref{th:Theret-elementary} is based on a series of length estimates of horogeodesic curves along anti-stretch lines, some of the estimates involving a decomposition of the surface into thin and thick parts, and using the work of Choi and Rafi in \cite{CR}.

  The following result by Th\'eret concerns the behavior of lengths of measured geodesic laminations under anti-stretch lines:

\begin{theorem}[Th\'eret, \cite{Theret-Thesis}, Theorem 6 p. 77]\label{th2}

Let $h:\mathbb{R}\to \mathcal{T}(S)$, $t\mapsto h_t$ be a stretch line passing through a hyperbolic metric $h=h(0)$ and supported by a geodesic lamination $\mu$. Let $\sigma$ be the stump of $\mu$ and assume that $\sigma$ is not empty. Then, for any measured geodesic lamination $\lambda$, we have:
\begin{enumerate}
\item $\lim_{t\to-\infty}l_{h_{t}}(\lambda)=0$ if and only if the support of $\lambda$ is contained in the support of $\sigma$.

\item If the support of $\lambda$ is not contained in the support of $\sigma$, then there exists a positive number $\epsilon(\lambda)$ such that for all $t\leq 0$, we have $l_{h_{t}}(\lambda)\geq \epsilon (\lambda)$. 

\item $\lim_{t\to-\infty}l_{h_{t}}(\lambda)=+\infty$ if and only if  $\lambda$ has non-empty transverse intersection with  $\sigma$.

\item $l_{h_{t}}(\lambda)$ is bounded for $t\leq 0$ if $\lambda$ disjoint from $\sigma$.
\end{enumerate}

\end{theorem}

The proof of (1) uses the fact that the length of the stump of $\mu$ is equal to its intersection number with the horocyclic foliation (or the geodesic lamination representing it). One applies then this formula to the hyperbolic structure $h_t$ for $t<0$ and finds a quantity that tends to 0 as $t\to-\infty$. The proof of (3) follows from (1) and the fact that when the length of a measured lamination converges to 0, the length of any measured lamination intersecting it transversely converges to infinity. This is well known in the case where the geodesic laminations are simple closed curves (it follows from the collar lemma), and it is proved in the general case in \cite[Theorem 5, p. 67]{Theret-Thesis}.

From Theorems \ref{th1} and \ref{th2}, Th\'eret deduces the following

\begin{corollary}
Let $h:\mathbb{R}\to \mathcal{T}(S)$ be a stretch line directed by a horocyclic geodesic lamination $L_{\mu}(h)$ and supported by a geodesic lamination with stump $\sigma$. Let $S'$ be a sub-surface of $S$ that 
does not intersect neither $L_{\mu}(h)$ nor $\sigma$. Then, along the stretch line $h(t)$, the subsurfaces corresponding to $S'$, equipped with their hyperbolic structures, are all quasi-isometric, with a quasi-isometry constant that does not depend on $t$.
\end{corollary}
\section{Concluding remarks}\label{s:conclusion}

The results that we presented on the asymptotic behavior of anti-stretch lines involve various hypotheses on the foliations or laminations: the stump being uniquely ergodic, or being a multicurve (elementary stretch line), the direction being a multicurve (cylindrical stretch line), etc. A study of the limiting behavior of other cases remains open.

 Let me also point out that there are many other classes of geodesics for Thurston's metric than stretch lines, some of them admitting explicit descriptions that generalize Thurston's stretch rays (see e.g. the geodesics in the papers \cite{2009j} and \cite{2015-Yamada1} based on stretch maps between right-angled hexagons, and those in the paper \cite{HP1} based on stretch maps between Saccheri ideal quadrilaterals). It would be interesting to study the asymptotic behavior of the various kinds of geodesics of Thurston's metric (possibly all geodesics), in the positive and negative direction, and see what results that are known for some specific types of stretch and anti-stretch lines, regarding parallelism and convergence to points on Thurston's boundary, hold in more general settings. 
 
 The study of the behavior of geodesics in Teichm\"uller space leads to the study of the geodesic flow on that space. One expects that the study of the geodesic flow associated with the Thurston metric, with the variety of geodesics that it involves and with the complications arising from the fact of the metric is non symmetric and that there exist many kinds of geodesics that join two points, will lead to several interesting developments.
 
 Finally, let me mention that very little is known on the geometry of the symmetrization of the Thurston metric (this applies to the max and the arithmetic symmetrizations).


\begin{thebibliography}{99}
     
   \bibitem{AD} D. Alessandrini and V. Disarlo,  Generalized stretch lines for surfaces with boundary, preprint, 2019.
    
    \bibitem{Busemann} H. Busemann, Recent synthetic differential geometry. Ergeb. Math. Grenzgeb. 54 (1970).
 
\bibitem{AP} N. A'Campo and A. Papadopoulos, Notes on non-Euclidean geometry. In: Strasbourg master class on geometry, p. 1--182, IRMA Lect. Math. Theor. Phys., 18, Eur. Math. Soc., Zürich, 2012.
 
 \bibitem{CEG} R. D. Canary, D. B. A. Epstein, and P. Green, Notes on notes of Thurston, In: Analytical and geometric aspects of hyperbolic space, Cambridge Univ. Press, 1987, p. 3--92.
 
 \bibitem{CB} A. Casson and S. Bleiler, Automorphisms of Surfaces after Nielsen and Thurston,  Cambridge University Press, 1988.
    
  \bibitem{CR} Y.-E. Choi and K. Rafi, Comparison between Teichmüller and Lipschitz  metric, J. London
Math. Soc. (2) 76 (2007), p. 739--756.
 
\bibitem{DGK} J. Danciger, F. Gu\'{e}ritaud, and F. Kassel, Margulis spacetimes via the arc complex, Invent. Math. 204 (2016), no.~1, p. 133--193.

 \bibitem{DLRT} D. Dumas, A. Lenzhen, K. Rafi, and J. Tao,  Coarse and fine geometry of the Thurston metric, Forum of Mathematics, Sigma, 8 (2020), Paper No/ e28, 58 pp.
 
 
\bibitem{FLP} A. Fathi, F. Laudenbach and V. Po\'enaru, Travaux de Thurston sur les surfaces. Ast\'erisque 66-67 (1979).


\bibitem{Hadamard} J. Hadamard, Les surfaces \`a courbures oppos\'ees et leurs lignes g\'eod\'esiques, J. Math. Pures
Appl. 4 (1898), p. 27--74. \OE uvres, tome II, p. 729--775.


\bibitem{HS} Y. Huang and Z. Sun, McShane identities for higher Teichm\"uller theory and the Goncharov--Shen potential, arxiv:1901.02032, 2018. 


\bibitem{HP1} Y. Huang, and  A. Papadopoulos, Optimal Lipschitz maps on one-holed tori and the Thurston metric theory of Teichm\"uller space, Geometriae Dedicata, to appear, 2021.


\bibitem{HP2} Y. Huang, and A. Papadopoulos, Optimal Lipschitz maps on bordered hyperbolic surfaces and the Thurston metric theory of Teichm\"{u}ller space, in preparation.


\bibitem{Kerckhoff} S. P. Kerckhoff, The asymptotic geometry of Teichm\"uller space. Topology 19 (1980), p. 23--41. 


\bibitem{LRT} A. Lenzhen, K. Rafi and J. Tao, The shadow of a Thurston geodesic to the curve graph, J. Topol. 8 (2015), p. 1085--1118.


\bibitem{Levitt}  G. Levitt, Foliations and laminations on hyperbolic surfaces. Topology 22 (1983), no. 2, p. 119--135. 



\bibitem{Liu} L. Liu, On the metrics of length spectrum in Teichm\"uller space, Chinese J. Contemp. Math. 22 (1) (2001), p. 23-34.
 
 
   \bibitem{Lobachevsky} N. I. Lobachevsky,  Pangeometry, Translation, notes and Commentary by A. Papadopoulos, Heritage of European Mathematics, Vol. 4, European Mathematics Publishing House,   2010.


 \bibitem{Masur} H. Masur, On a class of geodesics in Teichm\"uller space, Ann. of Math. 2. 102 (1975)
p. 205--221. 

\bibitem{1988} A. Papadopoulos,  Sur le bord de Thurston de l'espace de Teichm\"uller d'une surface 
non compacte, Mathematische Annalen,  282, (1988) p. 353--359.


\bibitem{1991a} A. Papadopoulos, On Thurston's boundary of Teichm\"uller space and the extension of earthquakes,  
Topology and its Applications, 41 (1991) p. 147--177.
  
\bibitem{2012e}  A. Papadopoulos and W. Su, On the Finsler structure of the Teichm\"uller and the Lipschitz metrics,  Expositiones Mathematicae, Elsevier,  33 (1) (2015), p. 30--47
 
      \bibitem{PT} A. Papadopoulos and G. Th\'eret,
On Teichm\"uller's metric and Thurston's asymmetric metric on Teichm\"uller space, in: Handbook of Teichm\"uller theory, ed. A. Papadopoulos, Vol. I,  p. 111--204,
IRMA Lect. Math. Theor. Phys., 11, Eur. Math. Soc., Z\"urich, 2007.


\bibitem{2007a} A. Papadopoulos and  G. Th\'eret, 
On the topology defined by Thurston's asymmetric metric,
Math. Proc. Camb. Philos. Soc. 142, No. 3 (2007), p. 487--496. 

 \bibitem{2009j}  A. Papadopoulos and  G. Th\'eret,  Some Lipschitz maps between hyperbolic surfaces with applications to Teichm\"uller theory,   Geometriae Dedicata, 150, 1 (2011) p. 233--247. 


\bibitem{Shift} A. Papadopoulos and  G. Th\'eret,  Shift coordinates, stretch lines and polyhedral structures for Teichm\"uller space, Monatsh. Math. 153  (2008), p. 309--346.


\bibitem{Lambert-Blanchard}  A. Papadopoulos and G. Th\'eret, La th\'eorie des lignes parall\`eles de Johann Heinrich Lambert,  critical edition of Lambert's ``Theorie der Parallellinien" with a French translation and mathematical and historical commentaries, ed. Blanchard, coll. Sciences dans l'Histoire, Paris, 2014.

     
    

\bibitem{2015-Yamada1} A. Papadopoulos and S. Yamada, Deforming Hexagons and the arc and the Thurston metric on Teichm\"uller space,  Monatshefte fur Mathematik 172(1) (2017) p. 97--120.


\bibitem{PH} R. C. Penner and J. L. Harer, Combinatorics of Train Tracks, Princeton University Press, 1992.


\bibitem{S} W. Su, Problems on the Thurston metric, in: 
 Handbook of Teichm\"uller theory, ed. A. Papadopoulos, Vol. V,  European Mathematical Society, Z\"urich, 2015, p. 55--72.


\bibitem{Teichmuller} O. Teichm\"uller, Extremale quasikonforme Abbildungen und quadratische Differentiale, Abh.
Preuss. Akad. Wiss. Preuss. Akad. Wiss. Math.-Nat. Kl. 1939 (1940). no. 22, 197 pp. 
English translation by G. Th\'eret, ``Extremal quasiconformal mappings and quadratic differentials,"
in A. Papadopoulos (ed.), Handbook of Teichmüller theory, Vol. V, IRMA Lectures in Mathematics
and Theoretical Physics 26, European Mathematical Society (EMS), Zürich, 2016, p. 321--483.


\bibitem{Theret-Thesis} G. Th\'eret, \`A propos de la m\'etrique asym\'etrique de Thurston sur l'espace de Teichm\"uller d'une surface.  Th\`ese, Universit\'e Louis Pasteur (Strasbourg I), Strasbourg, 2005. Pr\'epublication de l'Institut de Recherche Math\'ematique Avanc\'ee, 2005/8. Universit\'e Louis Pasteur. Institut de Recherche Math\'ematique Avanc\'ee (IRMA), Strasbourg, 2005. 96 pp.

\bibitem{T1} G. Th\'eret,   Divergence et parall\'elisme des rayons d'\'etirements cylindriques. Algebr. Geom. Topol. 10 (2010), no. 4, p. 2451--2468.
    
         
\bibitem{T3} G. Th\'eret,   On elementary anti-stretch lines. Geom. Dedicata 136 (2008), p. 79--93. 
 
\bibitem{T4} G. Th\'eret,  On the negative convergence of Thurston's stretch lines towards the boundary of Teichmüller space. Ann. Acad. Sci. Fenn. Math. 32 (2007), no. 2, p. 381--408.
 
\bibitem{Thurston1986} W. P. Thurston, Minimal stretch maps between hyperbolic surfaces, preprint, 1986,
Arxiv:math GT/9801039.



\bibitem{Thurston-Princeton} W. P. Thurston, The Geometry and Topology of Three-Manifolds. Lecture Notes, Princeton University, 1976.


\bibitem{Thurston-Book} W. P. Thurston, Three-dimensional Geometry and Topology, Vol. I, ed. Silvio Levy,  Princeton  Mathematical Series 35, Princeton University Press, 1997.

\bibitem{Thurston-FLP} W. P. Thurston,  On the geometry and dynamics of diffeomorphisms of surfaces. Bull. Amer. Math. Soc. (N.S.) 19 (1988), no. 2, p. 417--431.

  

    \end{thebibliography}
\end{document}